\documentclass[a4paper, reqno]{amsart}
\usepackage[T1]{fontenc}
\usepackage{amssymb,amsthm,amsmath}
\usepackage[british]{babel}
\usepackage{xcolor}
\usepackage[all]{xy}
\usepackage{calrsfs}
\usepackage{graphicx}
\usepackage{hyperref}
\usepackage{mathbbol}
\usepackage{mathabx}
\usepackage{tikz}

\usetikzlibrary{matrix,arrows}

\theoremstyle{plain}
\newtheorem{theorem}[subsection]{Theorem}
\newtheorem{lemma}[subsection]{Lemma}
\newtheorem{proposition}[subsection]{Proposition}
\newtheorem{corollary}[subsection]{Corollary}

\theoremstyle{definition}
\newtheorem{definition}[subsection]{Definition}

\theoremstyle{remark}
\newtheorem{example}[subsection]{Example}
\newtheorem{remark}[subsection]{Remark}

\newcommand{\K}{\mathbb{K}}

\newcommand{\Hom}{{\sf Hom}}

\newcommand{\set}{\ensuremath{\mathsf{Set}}}

\newcommand{\homlie}{\ensuremath{\mathsf{Hom\text{-}Lie}}}
\newcommand{\homset}{\ensuremath{\mathsf{Hom\text{-}Set}}}
\newcommand{\hommagma}{\ensuremath{\mathsf{Hom\text{-}Mag}}}
\newcommand{\multhomlie}{\ensuremath{\mathsf{Hom\text{-}Lie_{\rm mult}}}}
\newcommand{\Xhomlie}{\ensuremath{\mathsf{XHom\text{-}Lie}}}
\newcommand{\cathomlie}{\ensuremath{\mathsf{Cat}^1\mathsf{\text{-}Hom\text{-}Lie}}}

\DeclareMathOperator{\Id}{Id}

\DeclareMathOperator{\Der}{Der}
\DeclareMathOperator{\IDer}{IDer}
\DeclareMathOperator{\Ext}{Ext}
\DeclareMathOperator{\ab}{ab}

\DeclareMathOperator{\Ima}{Im}
\DeclareMathOperator{\Ker}{Ker}

\def\pullback{
	\ar@{-}[]+R+<6pt,-1pt>;[]+RD+<6pt,-6pt>%
	\ar@{-}[]+D+<1pt,-6pt>;[]+RD+<6pt,-6pt>}

\def\ophalfsplitpullback{
	\ar@{-}[]+R+<6pt,-1pt>;[]+RD+<6pt,-6pt>%
	\ar@{-}[]+D+<.5ex,-6pt>;[]+RD+<6pt,-6pt>}

\def\ophalfsplitpullbacktwo{
	\ar@{-}[]+R+<6pt,-4pt>;[]+RD+<6pt,-6pt>%
	\ar@{-}[]+D+<.5ex,-6pt>;[]+RD+<6pt,-6pt>}

\begin{document}

\title[Abelian extensions and crossed modules of Hom-Lie algebras]{Abelian extensions and crossed modules of Hom-Lie algebras}

\author{Jos\'e-Manuel Casas}
\address[Jos\'e-Manuel Casas]{Departamento de Matem\'atica Aplicada I, Universidade de Vigo, 36005 Pontevedra, Spain}
\email{jmcasas@uvigo.es}

\author{Xabier García-Martínez}
\address[Xabier García-Martínez]{Departamento de Matem\'aticas, Universidade de Santiago de Compostela, 15782 Santiago de Compostela, Spain.}
\email{xabier.garcia@usc.es}

\thanks{This work was partially supported by Ministerio de Economía y Competitividad (Spain), grant MTM2016-79661-P}

\begin{abstract}
	In this paper we study the low dimensional cohomology groups of Hom-Lie algebras and their relation with derivations, abelian extensions and crossed modules. On one hand, we introduce the notion of $\alpha$-abelian extensions and we obtain a five term exact sequence in cohomology. On the other hand, we introduce crossed modules of Hom-Lie algebras showing their equivalence with cat$^1$-Hom-Lie algebras, and we introduce $\alpha$-crossed modules to have a better understanding of the third cohomology group.
\end{abstract}


\maketitle

\section{Introduction}

Hom-Lie algebras were originally introduced in~\cite{HaLaSi} to study the deformations of the Witt and Virasoro algebras, mainly motivated by the study of quantum deformations and discretisation of vector fields via twisted derivations. A Hom-Lie algebra is an anti-commutative algebra satisfying a Jacobi identity twisted by a linear map. The linear map twisting the structure can be seen as a $1$-ary operation so we might think of non-associative structures over the abelian category of vector spaces with a chosen endomorphism (or it can even be done in the broader setting of monoidal categories~\cite{CaGo}). From a categorical algebraic point of view, the interest resides in that when we move to a Hom version of a well-known object, we lose many of their categorical properties, obtaining a richer and more complex structure.

Cohomology of Hom-Lie algebras was introduced in~\cite{Sheng} via a Chevalley-Eilenberg type complex, and the present paper is devoted to study its relations with abelian extensions and crossed modules. In classical categories (groups or Lie algebras), the first cohomology group represents the quotient of derivations by inner derivations, the second one represents abelian extensions~\cite{Wei, HiSt} and the third one, crossed extensions~\cite{Wei, KaLo, Rat}. In this paper we will see that the first and second assertions are not entirely true in Hom-Lie algebras, but may be fixed introducing the notion of $\alpha$-derivation and $\alpha$-abelian extension. On the other hand, the case of the third cohomology group and crossed modules is more complicated. We will introduce the notion of $\alpha$-crossed module to have a partial answer.

This paper is organised as follows. In Section~\ref{S:preliminaries} we recall some preliminaries on Hom-Lie algebras. In Section~\ref{free} we introduce the notion of free Hom-Lie algebra over a Hom-set. In Section~\ref{S:cohomology} we study actions, derivations and we recall the definition of cohomology. Section~\ref{S:extensions} is devoted to study abelian extensions of Hom-Lie algebras and their relations with the second cohomology group, obtaining a five term exact sequence. In Section~\ref{S:crossed} we define crossed modules of Hom-Lie algebras and study their relation with cat$^1$-Hom-Lie algebras and internal categories, illustrating that this definition is coherent with the usual notion of crossed module. Finally, in Section~\ref{S:third_cohomology} we study the third cohomology group of Hom-Lie algebras for which the introduction of $\alpha$-crossed modules is required.

\section{Preliminaries on Hom-Lie algebras}\label{S:preliminaries}

\begin{definition} \cite{HaLaSi} \label{def}
A \emph{Hom-Lie algebra} $(L,[-,-],\alpha_L)$ is a $\mathbb{K}$-vector space $L$ endowed with a bilinear map
$[-,-] \colon L \times L \to L$ and a homomorphism of $\mathbb{K}$-vector spaces $\alpha_L\colon L \to L$
satisfying:
\begin{align}
[x,y] &= - [y,x] \label{Skew} \tag{Skew-symmetry} \\
[\alpha_L(x),[y,z]]&+[\alpha_L(z),[x,y]]+[\alpha_L(y),[z,x]]=0 \label{Jacobi} \tag{Hom-Jacobi identity}
\end{align}
for all $x, y, z \in L$.

A Hom-Lie algebra $(L,[-,-],\alpha_L)$ is said to be
\emph{multiplicative} if the linear map $\alpha$ preserves the bracket~\cite{Yau2}. If we ask $\alpha$ to be a $\K$-algebra automorphism, then it is called \emph{regular}. We will use the shorter notation $(L,\alpha_L)$ whenever there is no confusion with the bracket
\end{definition}

\begin{example}\label{ejemplo 1} \hfil
\begin{enumerate}
\item[(a)] If $\alpha_L = \Id$, we recover the definition of Lie algebra. It is in particular multiplicative and regular.

\item[(b)] If $(A,\mu_A, \alpha_A)$ is a multiplicative Hom-associative algebra (see \cite{MaSi2} for a definition), then
   $HLie(A)$ $=(A,[-,-],\alpha_A)$ is a multiplicative Hom-Lie algebra, where
   $[x,y]=\mu_A(x,y)-\mu_A(y,x)$, for all $x,y \in A$ (see~\cite{Yau2}).

\item[(c)] Let $(L,[-,-])$ be a Lie algebra and $\alpha\colon L \to L$ an endomorphism of Lie algebras.
  If we define $[-,-]_{\alpha} \colon L \otimes L \to L$ by $[x,y]_{\alpha} = \alpha[x,y]$, for all $x, y \in L$,
  then $(L,[-,-]_{\alpha}, \alpha)$ is a multiplicative Hom-Lie algebra \cite[Theorem 5.3]{Yau2}.

\item[(d)] An \emph{abelian} Hom-Lie algebra is a $\K$-vector space $V$ endowed with trivial bracket and any
linear map $\alpha\colon V \to V$ (see~\cite{HaLaSi}).

\item[(e)] The Jackson Hom-Lie algebra $\mathfrak{sl}_2(\mathbb{K})$ is a Hom-Lie deformation of the classical Lie algebra $\mathfrak{sl}_2(\mathbb{K})$ defined by
$[h,f]=-2f, [h,e]=2e,[e,f]=h$. The Jackson $\mathfrak{sl}_2(\mathbb{K})$ is related to derivations. It is generated as a $\mathbb{K}$-vector space by $e, f, h$ with
multiplication given by $[h,j]_t=-2f-2tf, [h,e]_t=2e, [e,f]_t=h+ \frac{t}{2} h$ and the linear map $\alpha_t$ is defined by 
$\alpha_t(e) = \frac{2+t}{2(1+t)}e= e+ \displaystyle \sum_{k=0}^{\infty} \frac{(-1)^k}{2} t^k e, \alpha_t(h)=h, \alpha_t(f)=f+ \frac{t}{2} f$ (see~\cite{MaSi}).

  \item[(f)] For examples coming from deformations we refer to~\cite{Yau2}.

  \item[(g)] Two-dimensional complex multiplicative Hom-Lie algebras are classified in~\cite{CaInPa}.
\end{enumerate}
\end{example}

\begin{definition}\label{homo}
A \emph{homomorphism of Hom-Lie algebras} $f\colon(L,\alpha_L) \to
(L',\alpha_{L'})$ is a $\mathbb{K}$-linear map $f \colon L \to L'$ such that:
\begin{enumerate}
\item[(a)] $f([x,y]) =[f(x),f(y)]'$,
\item [(b)] $f \circ \alpha_L(x) = \alpha_{L'} \circ f(x)$,
\end{enumerate}
for all $x, y \in L$.
A \emph{homomorphism of multiplicative Hom-Lie algebras} is a homomorphism of the underlying Hom-Lie algebras.
\end{definition}

We denote by $\homlie$ (resp.\ $\multhomlie$) the category of Hom-Lie algebras (resp.\ $\multhomlie$).
There is an inclusion functor
 ${inc \colon \multhomlie \to \homlie}$ which has as left adjoint the multiplicative functor ${(-)_{\rm
mult} \colon \homlie \to \multhomlie}$ which assigns to a Hom-Lie algebra $(L,[-,-],\alpha_L)$ the multiplicative Hom-Lie algebra $(L/I,[-,-],\alpha_{\widetilde{L}})$, where $I$ is the ideal of $L$
generated by the elements $\alpha_L[x,y]-[\alpha_L(x),\alpha_L(y)]$, for all
 $x, y \in L$.

Both $\homlie$ and $\multhomlie$ are examples of semi-abelian categories (\cite{JaMaTh}), since they are pointed protomodular varieties of algebras (in particular varieties of $\Omega$-groups). Nevertheless, in the context of semi-abelian categories they do not satisfy some desired properties, such as ``Smith is Huq''~\cite{MaVa} (LAAC)~\cite{Gray2012} or algebraic coherence~\cite{acc}, and they are the main example of semi-abelian categories that do not satisfy the universal central extension condition~\cite{CaVdL}. On the other hand, we pleasantly find that the category of regular Hom-Lie algebras is isomorphic to Lie objects over a certain symmetric monoidal category~\cite{GoVe}.

Therefore, even though the general notion of Hom-Lie algebras is more permissive, for a comprehensive structural study we need the stability properties that the multiplicative condition offers, so in the sequel {\bf we will say Hom-Lie algebra when referring multiplicative Hom-Lie algebra}.

\begin{definition}
Let $(L,\alpha_L)$ be a Hom-Lie algebra. A \emph{Hom-Lie subalgebra} $(H,\alpha_H)$ of $(L,\alpha_L)$
 is a vector subspace $H$ of $L$, which is closed by the bracket and is invariant by $\alpha_L$, that is,
\begin{enumerate}
\item [(a)] $[x,y] \in H,$ for all $x, y \in H$,
\item [(b)] $\alpha_L(x) \in H$, for all $x \in H$.
\end{enumerate}

A Hom-Lie subalgebra $(H,\alpha_H)$ of $(L,\alpha_L)$ is said to be a \emph{Hom-ideal} if $[x,y] \in H$
for all $x \in H, y \in L$.

If $(H,\alpha_H)$ is a Hom-ideal of $(L,\alpha_L)$, then $(L/H,\overline{\alpha_L})$, where $[\overline{x},\overline{y}] = \overline{[x, y]}$, for all $\overline{x}, \overline{y} \in L/H$ and $\overline{\alpha_L} \colon L/H \to L/H$ is naturally induced by $\alpha_L$, inherits
a Hom-Lie algebra structure, which is named \emph{quotient Hom-Lie algebra}. An example of Hom-ideal is the \emph{commutator ideal} $[L, L]$, which is the subalgebra of
$(L, \alpha_L)$ generated by all the elements $[x, y]$, for all $x, y \in L$. Moreover, the quotient $(L/[L, L], \overline{\alpha_L})$ is called the abelianisation of $(L, \alpha_L)$ and it is denoted by $(L^{\ab}, \alpha_{L^{\ab}})$.

\end{definition}

\begin{definition}
We call \emph{centre} of a Hom-Lie algebra $(L,\alpha_L)$ to the Hom-Lie subalgebra
\[
Z(L) = \{ x \in L \mid [\alpha^{n}(x), y] =0 , \text{for all}\ y \in L, n \geq 0 \}
\]
When $\alpha_L \colon L \to L$ is an epimorphism, then $Z(L)$ is a Hom-ideal of $L$. This definition coincides with the categorical notion of centre introduced by Huq~\cite{Huq}.
\end{definition}

\section{Free Hom-Lie algebras} \label{free}

The construction of free Hom-Lie algebras given immediately below is the generalization of the construction of a free Lie algebra over a set given in~\cite{Serre}. An alternative construction can be found in~\cite{Yau}.

\begin{definition} The category whose objects are pairs $(X,\alpha_{X})$, where $X\in \set$ and $\alpha_{X}\colon X\to X$ is a map, and whose morphisms are maps $f\colon(X,\alpha_{X}) \to (Y,\alpha_{Y})$, where $f\colon X\to Y$ is a map such that $f\circ \alpha_{X}=\alpha_{Y}\circ f$, is called the \emph{category of Hom-sets} and it is denoted by $\homset$.
\end{definition}

\begin{definition}
	Given a Hom-set $(M,\alpha_{M})$ we may endow it with a Hom-magma operation, i.e.\ a magma multiplication $w\colon M\times M\to M$ on the set $M$, such that ${w \circ (\alpha_{M}\times\alpha_{M}) =\alpha_{M} \circ w}$. A morphism of Hom-magmas ${f\colon (M,\alpha_{M}) \to (N,\alpha_{N})}$ is just a map that preserves the binary operation, i.e.\ $f (w_{M}(x,y)) = w_{M} (f(x),f(y))$, and $f \circ \alpha_{M}=\alpha_{N} \circ f$. The category of \emph{Hom-magmas} will be  denoted by $\hommagma$.
\end{definition}

To construct the free Hom-Magma over a Hom-set $(X,\alpha_{X})$, we define inductively the family of Hom-sets $(X_{n},\alpha_{n})$, $(n\geqslant1)$, as follows:
\newpage
\begin{enumerate}
\item[(i)] $X_{1}=X$; $\alpha_{1}=\alpha_{X}$.

\item[(ii)] $X_{n}=\underset{p+q=n}{\coprod}X_{p}\times X_{q}$ $(
n\geqslant2) $ (disjoint union); $\alpha_{n}\colon X_{n}\longrightarrow
X_{n}$ induced by $\alpha_{X}$ (thus $\alpha_{2}\colon X_{2}(=X\times X) \longrightarrow
X_{2}(=X\times X)$, where $\alpha_{2}=\alpha_{X}\times
\alpha_{X}$; $\alpha_{3}=\alpha_{X}\times(\alpha_{X}\times\alpha_{X})
+(\alpha_{X}\times\alpha_{X}) \times\alpha_{X}$, etc.).
\end{enumerate}

Take $M_{X}=\underset{n=1}{\overset{\infty}{\coprod}}X_{n}$ and $\alpha
_{M}\colon M_{X}\to M_{X}$ the endomorphism induced by $\alpha_{X}$, i.e.\
$\alpha_{M}=(\alpha_{1}, \alpha_{2}, \alpha_{3}, \dots)$.
We define the binary operation $w\colon M_{X}\times M_{X}\to M_{X}$ by means of $X_{p}\times X_{q}\longrightarrow X_{p+q}\subseteq M_{X}$, where this map is the canonical inclusion induced by (ii).

\begin{definition} 
$(M_{X},\alpha_{M})$ is said to be the \emph{free Hom-Magma} over the Hom-set $(X,\alpha_{X})$.
An element $\gamma$ de $M_{X}$ is said to be a \emph{non-associative word} over $X$. Its length, $l(\gamma)$, is the unique $n$ such that $\gamma\in X_{n}$.
\end{definition}

\begin{theorem}\label{magma libre Lie}
Let $(N,\alpha_{N})$ be a Hom-Magma and $f\colon (X,\alpha_{X}) \to (N,\alpha_{N})$
a morphism of Hom-sets.
Then there exists a unique homomorphism of Hom-magmas $F\colon (M_{X},\alpha_{M}) \to (N,\alpha_{N})$ that
extends $f$.
\end{theorem}

\begin{proof}
We define by induction $F\colon M_{X}\to N$  as follows: any word $\epsilon \in X_n$ can be decomposed in words of lower length $\gamma \in X_p, \delta \in X_q$, with $p+q=n$; then $F(\epsilon) = F(w_M(\gamma, \delta)) =w_{N}(F(\gamma),F(\delta))$, and when $\epsilon \in X_1$, then $F(\epsilon)=f(\epsilon)$. Here $w_N$ denotes the binary operation in the Hom-Magma $(N,\alpha_N)$.
\end{proof}

Theorem \ref{magma libre Lie} provides the adjoint functors
\[
\xymatrix{ 
\homset \ar@<1ex>[r]^{\mathcal{F}_1} \ar@{}[r]|-{\bot} & \ar@<1ex>[l]^{\mathcal{U}_1} \hommagma
}
\]
$\mathcal{F}_1$ assigns to a Hom-set $(X, \alpha_X)$ the free Hom-magma $\mathcal{F}_1(X, \alpha_X) = (M_X,\alpha_M)$ and $\mathcal{U}_1$ assigns to a Hom-magma $(N,\alpha_N)$ the Hom-set obtained by forgetting the binary operation.

\begin{definition} \cite{Yau} A \emph{non-associative Hom-algebra} is a triple $(A, \mu, \alpha)$ where:
\begin{enumerate}
\item[(a)] $A$ is a $\mathbb{K}$-vector space.

\item[(b)] $\mu\colon A\otimes A\to A$ is a bilinear map.

\item[(c)] $\alpha\in End(A)$ preserves the bilinear map, i.e.\
$\alpha(\mu(x\otimes y))
=\mu(\alpha(x) \otimes\alpha(y))$, for all $x, y \in A$.
\end{enumerate}

A morphism between two non-associative Hom-algebras $f\colon (A,\mu,\alpha) \to (A^{\prime},\mu^{\prime},\alpha^{\prime
})$ is a linear map $f\colon A\to A^{\prime}$ such that $\alpha^{\prime} \circ f=f \circ \alpha$ and $f \circ \mu=\mu^{\prime} \circ f^{\otimes2}$.
\end{definition}

Let $(A_{X},\alpha_{A})$ be such that $A_{X}$ is the $\mathbb{K}$-algebra generated by the free magma $M_{X}$ and $\alpha_{A}$ the endomorphism induced by $\alpha_{X}$. More explicitly, an element $\gamma\in A_X$ is a finite sum $\gamma=\displaystyle \sum_{m\in M_{X}} c_{m} m$, where $c_{m}\in \mathbb{K}$; the multiplication in $A_{X}$ extends the
multiplication (binary operation) in $M_{X}$ and $\alpha_{A}\colon A_{X}\to A_{X}$ is given by ${\alpha_{A}(\gamma)
=\displaystyle \sum_{m\in M_{X}} c_{m} \alpha_{x}(m)}$.

In this way, $(A_{X},\alpha_{A})$ is a non-associative Hom-algebra where $A_{X}$ is a $\mathbb{K}$-vector space spanned by $X$, the binary operation $\mu$ is derived from the operation $\omega$ of the magma and $\alpha_A$ preserves this operation.

Moreover $(A_{X},\alpha_{A})$ is the free Hom-algebra over $(X,\alpha_{X})$, due to the following property.

\begin{theorem} \label{algebra libre Lie} Let $(B,\alpha_{B})$ be a Hom-algebra and
 $f\colon (X,\alpha_{X}) \to (B,\alpha_{B})$ a map of Hom-sets.
There exists a unique homomorphism of Hom-algebras $F\colon (A_{X},\alpha_{A}) \to (B,\alpha_{B})$ that extends $f$.
\end{theorem}

\begin{proof}
	By Theorem \ref{magma libre Lie}, $f$ can be extended to a homomorphism of Hom-magmas $f^{\prime}\colon (M_{X},\alpha_{M}) \to (B,\alpha_{B})$, and $f'$ extends by linearity to a $\mathbb{K}$-linear morphism ${F\colon A_X \to B}$. By construction, $F$ satisfies the required properties.
\end{proof}

\begin{remark}
$A_X$ is a graded algebra whose homogeneous elements of degree $n$ are linear combinations of words $m \in M_X$ of length $n$. The endomorphism $\alpha_A$ is induced by $\alpha_X$.
\end{remark}

Let $I$ be the two-sided ideal of $(A_{X},\alpha_{A})$ spanned by the elements of the form:
\begin{enumerate}
\item[(a)] $ab+ba$,
\item[(b)] $\alpha_{A}(a) (
bc) +\alpha_{A}(c) (ab) +\alpha_{A}(
b) (ca)$,
\end{enumerate}
where  $a$,$b$,$c\in A_{X}$.

\begin{definition}\label{free algebra}
The quotient algebra $(A_{X}/I,\overline{\alpha}_{A})$ is said to be the \emph{free Hom-Lie algebra} over
the Hom-set $(X,\alpha_{X})$.
We denote it by $\mathcal{F}_r(X,\alpha_{X})$.
\end{definition}

\begin{theorem} \label{Hom-Lie libre}
Let $(B,\alpha_{B})$ be a Hom-Lie algebra and $f\colon (X,\alpha_{X}) \to (B,\alpha_{B})$
a map of Hom-sets. Then there exists a unique homomorphism of Hom-Lie algebras
$F\colon \mathcal{F}_{r}(X,\alpha_{X}) \to (B,\alpha_{B})$
that extends $f$.
\end{theorem}
\begin{proof} 
We know by Theorem \ref{algebra libre Lie} that there exists a unique homomorphism of algebras $f' \colon (A_X, \alpha_A) \to (B, \alpha_B)$ such that $f'\circ i = f$, where $i \colon (X,\alpha_{X}) \to (A_X, \alpha_A)$ is the canonical inclusion. Since $f^{\prime}$ vanishes over $I$, it induces a homomorphism of Hom-Lie algebras $\varphi \colon (A_X/I, \overline{\alpha_A}) \to (B, \alpha_B)$ that extends $f'$. 
\end{proof}

Theorem \ref{Hom-Lie libre} provides the adjoint functors
\[
\xymatrix{ \homset \ar@<1ex>[r]^{\mathcal{F}_r} \ar@{}[r]|-{\bot} & \ar@<1ex>[l]^{\mathcal{U}} \homlie}
\]
where $\mathcal{F}_r$ sends a Hom-set $(X, \alpha_X)$ to the free Hom-Lie algebra $\mathcal{F}_r(X, \alpha_X) = (A_X/I,\overline{\alpha_A})$ and $\mathcal{U}$ sends a Hom-Lie algebra $(B,\alpha_B)$ to the Hom-set obtained by forgetting the operations.

\begin{remark}
The above construction with $\alpha = \Id$ recovers the construction of the free Lie algebra over a set given in \cite{Serre}.
\end{remark}

\begin{remark}
	In a parallel way, one can also construct the free Hom-Lie algebra over a set, which is the standard notion of free algebra when seeing Hom-Lie algebras as vector spaces with two operations satisfying certain identities.
\end{remark}

\section{A Cohomology theory of Hom-Lie algebras}\label{S:cohomology}

\subsection{Actions}

\begin{definition}
Let $(L, \alpha_L)$ and $(M,\alpha_M)$ be Hom-Lie algebras. A \emph{Hom-action} from $(L,\alpha_L)$ over $(M,\alpha_M)$ consists in a bilinear map $\sigma \colon L \otimes M \to M, \sigma(x\otimes m) =x \centerdot m$ satisfying the following identities:
\begin{enumerate}
\item[(a)] $[x,y] \centerdot \alpha_M(m) = \alpha_L(x) \centerdot (y \centerdot m) - \alpha_L(y) \centerdot (x \centerdot m)$,

\item[(b)] $\alpha_L(x) \centerdot [m, m'] = [x \centerdot m, \alpha_M(m')]+ [\alpha_M(m),x \centerdot m']$,

\item[(c)] $\alpha_M(x \centerdot m) = \alpha_L(x) \centerdot \alpha_M(m)$,
\end{enumerate}
for all $x, y \in L$ and $m, m' \in M$. This definition coincides with the standard notion of internal action~\cite{Beck, BoJaKe2}.
Note that if we were in the non-multiplicative case, the right notion of action would be formed dropping identity (c). 

If there is a Hom-action from $(L,\alpha_L)$ over $(M,\alpha_M)$ and $(M,\alpha_M)$ is an abelian Hom-Lie algebra, then $(M,\alpha_M)$ is said to be a \emph{Hom-$L$-module} (see \cite{Yau2}).
\end{definition}

\newpage
\begin{example}\label{ex action} \hfil
\begin{enumerate}
\item[(a)] Let $M$ be a Lie $L$-algebra with action given by $\sigma \colon L \to {\rm Der}(M)$. Then there is a Hom-action from $(L,\Id_L)$ over $(M,\Id_M)$ given by $\sigma$.

\item[(b)] Let $(H,\alpha_H)$ be a Hom-subalgebra of a Hom-Lie algebra $(L,\alpha_L)$ and $(K,\alpha_K)$ is a Hom-ideal of $(L,\alpha_L)$. Then the bracket in $(L,\alpha_L)$ yields a Hom-action over $(H,\alpha_H)$ on $(K,\alpha_K)$.

\item[(c)] If we consider a  Lie algebra $\mathfrak{g}$ as a
Hom-Lie algebra in the way of  Example~\ref{ejemplo 1}~(a) and 
$M$ is a $\mathfrak{g}$-module in the usual way, 
then $(M, \Id_M)$ is a Hom-$\mathfrak{g}$-module.

\item[(d)] Let $\mathfrak{g}$ be a Lie algebra, $\alpha \colon \mathfrak{g} \to \mathfrak{g}$ an endomorphism and $M$ a $\mathfrak{g}$-module in the usual way, such that the action from $\mathfrak{g}$ over $M$ satisfies the condition $\alpha(x) \centerdot m = x \centerdot m$, for all $x \in \mathfrak{g}$ and $m \in M$. Then $(M, \Id)$ is a Hom-$\mathfrak{g}$-module. An example of this situation is given by the two-dimensional Lie algebra $L$ spanned by $\{ e, f \}$, with bracket $[e,f]=-[f,e]=e$ and endomorphism $\alpha$ represented by the matrix $\left(\begin{array}{cc} 1 & 1 \\ 0 & 1 \end{array}\right)$, where $M$ is the ideal spanned by $\{e\}$.

\item[(e)] Further examples can be found in~\cite[Example 6.2]{Yau2}.
\end{enumerate}
\end{example}

As expected, the notion of action induces the following definition:
\begin{definition}\label{D:semidirect}
	Let $(L,\alpha_L)$ and $(M,\alpha_M)$ be Hom-Lie algebras with an action from $(L,\alpha_L)$ over $(M,\alpha_M)$. The \emph{semi-direct product} $(M \rtimes L, \widetilde{\alpha})$ is the Hom-Lie algebra with underlying
	$\mathbb{K}$-vector space $M \oplus L$, with bracket
	\[
	[(m_1,x_1),(m_2,x_2)]= ([m_1,m_2]+ x_1 \centerdot m_2 - x_2 \centerdot m_1, [x_1,x_2])
	\]
	and endomorphism
	\[
	\widetilde{\alpha} \colon M \oplus L \to M \oplus L, \quad 
	\widetilde{\alpha} (m, x) = (\alpha_M(m), \alpha_L(x))
	\] 
	for all
	$x,x_1,x_2 \in L$ and $m, m_1, m_2 \in M$.
\end{definition}

There is an injective homomorphism $i \colon M \to M \rtimes L$ given by $i(m) =
(m,0)$ and a surjective homomorphism $\pi \colon M \rtimes L \to L$ given by
$\pi(m,l)= l$. Moreover $i(M)$ is a Hom-ideal of $M \rtimes L$
such that $\frac{M \rtimes L}{i(M)} \cong L$. On the other hand, there is a morphism $\sigma \colon L \to M \rtimes L$, $l \mapsto \sigma(l) = (0, l)$, which is clearly a section of $\pi$. Hence we obtain the split exact sequence of Hom-Lie algebras
\begin{equation}\label{extension semidirecto Lie}
\xymatrix{
0 \ar[r] & (M,\alpha_M) \ar[r]^-i & (M \rtimes L,\widetilde{\alpha}) \ar@<-.5ex>[r]_-{\pi} & (L,\alpha_L) \ar@<-.5ex>[l]_-{\sigma} \ar[r] & 0 
}
\end{equation}
and $(M,\alpha_M)$ is a Hom-${M \! \rtimes \! L}$-module via
$\pi$.

\begin{definition}\label{D:derivation}
Let $(L,\alpha_L)$ be a Hom-Lie algebra and $(M,\alpha_M)$
a Hom-$L$-module. A \emph{derivation} from $(L,\alpha_L)$ to $(M,\alpha_M)$ is a
$\mathbb{K}$-linear map $d \colon L \to M$ satisfying
\begin{align*}
d[x,y] &= x \centerdot d(y) - y \centerdot d(x), \\
\alpha_M \circ d &= d \circ \alpha_L,
\end{align*}
for all $x, y \in L$.

We denote by $\Der(L, M)$ the $\mathbb{K}$-vector space of all
derivations from $(L, \alpha_{L})$ to $(M,\alpha_{M})$.
\end{definition}

\begin{example}\hfill
\begin{enumerate}
\item [(a)] If $d \colon L \to M$ is a derivation of Lie algebras, then $d \colon (L,\Id_L) \to (M,\Id_M)$ is a derivation of Hom-Lie algebras.

\item[(b)] Let $(M, \alpha_M)$ be a Hom-$L$-module and consider the extension 
\[
\xymatrix{
	0 \ar[r] & (M, \alpha_M) \ar[r]^-{i} & (M \rtimes L, \widetilde{\alpha}) \ar[r]^-{\pi} &(L, \alpha_L) \ar[r] & 0.
}
\]
The Hom-linear map $\theta \colon M \rtimes L \to M$, $\theta(m, l) = m$ is a derivation.
\end{enumerate}
\end{example}

\begin{definition}
Let $s \colon (L, \alpha_L) \to (L, \alpha_L)$ be a homomorphism of Hom-Lie algebras and $M$ a Hom-$L$-module. Then, an \emph{$s$-derivation} is a Hom-linear homomorphism $d \colon (L, \alpha_L) \to (M, \alpha_M)$ making the following diagram commutative:
\[
\xymatrix{
 & (L, \alpha_L) \ar[dl]_-{d} \ar[rd]^-{s} \ar[d]^{h} & \\
 (M, \alpha_M) \ar@<.5ex>[r]^-{i} & (M \rtimes L, \widetilde{\alpha}) \ar@<.5ex>[l]^-{\theta} \ar[r]^-{\pi} & (L, \alpha_L)
}
\]
where $h(x) = \big(d(x), s(x)\big)$.
\end{definition}
As expected, $\Id_{L}$-derivations are just derivations (Definition~\ref{D:derivation}). We  denote the set of \mbox{$s$-derivations} by $\Der_{s}(L, M)$. Moreover, we define the \emph{$s$-semi-direct product} $M \rtimes_{s} L$ as the usual semi-direct product (Definition~\ref{D:semidirect}), but with a small modification in the bracket:
\[
[(m_1,x_1),(m_2,x_2)]= ([m_1,m_2]+s(x_1) \centerdot m_2 - s(x_2) \centerdot m_1, [x_1,x_2])
\] 
for all $x_1,x_2 \in L$ and $m_1, m_2 \in M$.

If we fix $s = \alpha_L$, then we obtain the notion of $\alpha$-derivations, which are exactly $\K$-linear maps $d \colon L \to M$ satisfying
\begin{align*}
d[x,y] &= \alpha_L(x) \centerdot d(y) - \alpha_L(y)
\centerdot d(x), \\
\alpha_M \circ d &= d \circ \alpha_L,
\end{align*}
for all $x, y \in L$.  

\begin{example}
	For a Hom-$L$-module  $(M,\alpha_M)$ and a fixed element $m \in M$ such that $m = \alpha_M(m)$, we define $d_m \colon L \to M$ as $d_m(x) = \alpha_L(x) \centerdot m$. 
	We call this kind of maps as \emph{inner $\alpha$-derivations} from $L$ to $M$ and we denote the set of all these $\alpha$-derivations by $\IDer_{\alpha}(L,M)$. An example of a Hom-$L$-module satisfying $m = \alpha_M(m)$ for any element $m \in M$ is given in Example~\ref{ex action}~(d).
\end{example}

\subsection{Cohomology} \hfill

Let $(L, \alpha_{L})$ be a Hom-Lie algebra and $(M,\alpha_{M})$ be a Hom-$L$-module. For $n\geq 0$, we denote by
\[
C_{\alpha}^{n}(L,M):=\{f\colon \bigwedge^{n}L\rightarrow M \mid f \text{ is a } \mathbb{K}\text{-linear map s.~t.\ }f \circ \alpha_{L}^{\wedge n}=\alpha_{M} \circ f\}
\]
the $\K$-vector space of $n$-cochains of $(L, \alpha_{L})$ with coefficients in $(M,\alpha_{M})$, where $\alpha_L^{\wedge n}(x_1\wedge \dots \wedge x_n) = \alpha_L(x_1) \wedge \dots \wedge \alpha_L(x_n)$.

For $n\geq 0$, we consider the $\mathbb{K}$-linear map,
\[
d^{n}\colon C_{\alpha}^{n}(L,M)\rightarrow C_{\alpha}^{n+1}(L,M)
\]
\begin{align*}
{} & d^{n}(f)(x_{1}\wedge\dots\wedge x_{n+1}) = \sum_{i=1}^{n+1}(-1)^{i+1}\alpha_{L}^{n}(x_{i})\centerdot f(x_{1}\wedge
\dots\wedge\widehat{x_{i}}\wedge\dots\wedge x_{n+1}) \\
{} &+ \sum_{i<j}(-1)^{i+j}f([x_{i},x_{j}]\wedge\alpha_{L}(x_{1})\wedge\dots
\wedge\widehat{\alpha_{L}(x_{i})}\wedge\dots\wedge\widehat{\alpha_{L}(x_{j}%
)}\wedge\dots\wedge\alpha_{L}(x_{n+1}))
\end{align*}
for all $m\in M$, $x_{i}\in L, i = 1, \dots, n+1$, with the hat denoting the omission of the respective term.

The identity $d^{n+1}\circ d^{n}=0$, holds for $n\geqslant0$ (see~\cite{Sheng}), consequently
$(C_{\alpha}^{*}(L,M),d^{*})$ is a well-defined cochain complex whose homology is called the \emph{cohomology of the
 Hom-Lie algebra $(L,\alpha_{L})$} with coefficients in the Hom-$L$-module $(M,\alpha_{M})$, and it is denoted by:
\[
H_{\alpha}^{*}(L,M):=H^{*}(C_{\alpha}^{*}(L,M),d^{*})
\]
We denote the cocycles, i.e.\ the elements of $\Ker (d^n)$, by $Z^n_{\alpha}(L, M)$, and the coboundaries or elements of $\Ima (d^{n-1})$, by $B_{\alpha}^{n}(L,M)$.

For $n=0$, a direct computation shows that
\[
H_{\alpha}^{0}(L,M)=\{ m\in M\mid \alpha_M(m) = m \text{ and } x\centerdot m = 0, \text{for all}\ x \in L\} = {^L}M.
\]
It is the subspace of the elements of $M$ that are $L$-invariants.

For $n=1$, 1-cocycles are $\mathbb{K}$-linear maps
$f\colon L\to M$ satisfying the identities
\begin{align*}
f\left[ x_{1},x_{2}\right] &= \alpha_{L}(x_{1}) \centerdot f(x_{2}) -\alpha_{L}(x_{2}) \centerdot f(x_{1}) \\
f \circ \alpha_{L} &= \alpha_{M}\circ f
\end{align*}
i.e.\ $f$ is an $\alpha$-derivation of $(L,\alpha_{L})$ to
$(M,\alpha_{M})$.
Coboundaries are exactly inner derivations. Therefore,  
\[
H^1_{\alpha} (L, M) \cong \dfrac{\Der_{\alpha}(L, M)}{\IDer_{\alpha}(L, M)}
\]

\section{Abelian extensions}\label{S:extensions}

An \emph{abelian extension} of Hom-Lie algebras $(E)$ is an exact sequence of Hom-Lie algebras $\xymatrix@C=1em{0 \ar[r] & (M,\alpha_M) \ar[rr]^{i} && (E,\alpha_E) \ar@<-.5ex>[rr]_{\pi} && (L,\alpha_L) \ar[r] \ar@<-.5ex>[ll]_{\sigma} & 0,}$ where $(M,\alpha_M)$ is an abelian Hom-Lie algebra, $i$ and $\pi$ are homomorphisms of Hom-Lie algebras and $\sigma$ is a Hom-linear section of $\pi$.

\begin{remark}
	We might find the case when a surjective homomorphism of Hom-vector spaces does not have a section, as we can see in the following examples. 
	Let $X$ be the $\K$-Hom-vector space generated by $x_1$, $x_2$ with $\alpha_X(x_1) = 0$, $\alpha_X(x_2) = x_1$ and let $Y$ be the $\K$-Hom-vector space generated by $y$ with $\alpha_Y(y) = 0$. Let us consider the surjective homomorphism $\pi \colon (X, \alpha_X) \to (Y, \alpha_Y)$, $\pi(x_1) = 0$, $\pi(x_2) = y$. If it had a Hom-linear section $\sigma$, it would mean that $y = \pi \circ \sigma(y) = \pi(\lambda_1 x_1 + \lambda_2 x_2) = \lambda_2 x_2$, forcing $\sigma(y) = \lambda_1 x_1 + x_2$. However,
	\[
	0 = \sigma \alpha_Y(y)	 = \alpha_X \sigma(y) = \lambda_1 \alpha_X(x_1) + \alpha_X(x_2) = x_1,
	\]
	obtaining a contradiction.
	
	Let $X$ be the $\K$-Hom-vector space generated by $\mathbb{N}$ generators $x_1$, \dots, where $\alpha_{X}(x_n) = x_{n+1}$, and let $Y$ be the $\K$-Hom-vector space with basis $y$ and $\alpha_Y = \Id_Y$. Let us consider the morphism $\pi(x_i) = y$. If it had a section $\sigma \colon Y \to X$, it would be of the form $\sigma(y) = \sum_{i > 0} k_ix_i$, where just a finite number of $k_i$ are different from zero. Then 
	\begin{align*}
	\sum_{i > 0} k_ix_i = \sigma(y) = \sigma(\alpha_Y(y)) = \alpha_X(\sigma(y)) = \sum_{i > 0} k_i\alpha_X(x_i) = \sum_{i > 0} k_ix_{i+1}.
	\end{align*}
	Since the $x_i$ are linearly independent, there is a contradiction.
\end{remark}

In general, an abelian extension induces a Hom-action from $(L,\alpha_L)$ on $(M,\alpha_M)$ defined as $x \centerdot m = [\sigma(x),i(m)]$, $x \in L$ and $m \in M$, and an easy computation shows that it does not depend on the choice of $\sigma$.

Fixing $(M, \alpha_M)$ and $(L, \alpha_L)$, we say that two abelian extensions $(E)$ and $(E')$ are \emph{equivalent} if there exists a homomorphism of Hom-Lie algebras $\varphi\colon (E, \alpha_E) \to (E', \alpha_{E'})$ making the following diagram commutative:
\[
\xymatrix{
	0 \ar[r]& (M,\alpha_{M}) \ar[r]^-{i} \ar@{=}[d] & (E,\alpha_{E}) \ar@<-.5ex>[r]_-{\pi} \ar[d]_{\varphi} & (L,\alpha_{L}) \ar[r] \ar@<-.5ex>[l]_-{\sigma} \ar@{=}[d]& 0\\
	0 \ar[r]& (M,\alpha_{M}) \ar[r]^-{i'} & (E',\alpha_{E'}) \ar@<-.5ex>[r]_-{\pi'} & (L,\alpha_{L}) \ar[r] \ar@<-.5ex>[l]_-{\sigma'} & 0
}
\]
Protomodularity (see~\cite{Borceux-Bourn}) implies that $\varphi$ is and isomorphism and therefore, this relation is indeed an equivalence relation.
Note that we do not need to worry about the choice of the section, since two extensions with the same $i$ and $\pi$ but different section $\sigma$, are always equivalent.

Let us now fix $M$ to be a Hom-$L$-module and $s\colon (L, \alpha_L) \to (L, \alpha_L)$ to be a homomorphism of Hom-Lie algebras. We say that an abelian extension $(E)$ is in particular an \emph{abelian $s$-extension} if $[\sigma(x), i(m)] = s(x)\centerdot m$.
We denote by $\Ext_{s}(L, M)$ the set of equivalence classes of $s$-abelian extensions.
An example of $s$-abelian extension is 
\[
\xymatrix{
	0 \ar[r] & (M,\alpha_M) \ar[r]^-{i} & (M \rtimes_{s} L,\widetilde{\alpha}) \ar@<-.5ex>[r]_-{\pi} & (L,\alpha_L) \ar[r] \ar@<-.5ex>[l]_-{\sigma} & 0
}
\]
and it will be called the \emph{trivial $s$-extension}.

Let $(E)$ be an $s$-abelian extension and let $\gamma \colon (L', \alpha_{L'}) \to (L, \alpha_L)$ and $s' \colon (L', \alpha_{L'}) \to (L', \alpha_{L'})$ be homomorphisms of Hom-Lie algebras such that $\gamma \circ s' = s \circ \gamma$. Since pullbacks preserve split extensions and kernels, we define the \emph{backward induced extension} by 
\[
\xymatrix{
	0 \ar[r]& (M,\alpha_{M}) \ar[r] \ar@{=}[d] & (E_{\gamma},\alpha_{E_{\gamma}}) \ophalfsplitpullbacktwo \ar@<-.5ex>[r] \ar[d] & (L',\alpha_{L'}) \ar[r] \ar@<-.5ex>[l] \ar[d]^-{\gamma}& 0\\
	0 \ar[r]& (M,\alpha_{M}) \ar[r]^-{i}  & (E,\alpha_{E}) \ar@<-.5ex>[r]_-{\pi} & (L,\alpha_{L}) \ar[r] \ar@<-.5ex>[l]_-{\sigma} & 0
}
\]
where $E_{\gamma}$ is the pullback of $\pi$ and $\gamma$. Then, the homomorphism $\gamma$ induces a map $\gamma^{*} \colon \Ext_{s}(L, M) \to \Ext_{s'}(L', M)$.

On the other hand, for any morphism of Hom-$L$-modules $\delta \colon (M,\alpha_{M}) \to (M',\alpha_{M'})$ together with a morphism of Hom-Lie algebras $s' \colon (E,\alpha_{E}) \to (E,\alpha_{E})$ such that $\pi \circ s' = s \circ \pi$, we define the \emph{forward induced extension} by 
\[
\xymatrix{
	0 \ar[r]& (M,\alpha_{M}) \ar[r]^-{i} \ar[d]_{\delta} & (E,\alpha_{E}) \ar@<-.5ex>[r]_-{\pi} \ar[d] & (L,\alpha_{L}) \ar[r] \ar@<-.5ex>[l]_-{\sigma} \ar@{=}[d] & 0\\
	0 \ar[r]& (M',\alpha_{M'}) \ar[r]  & (^{\delta}E, \alpha_{(^{\delta}E)}) \ar@<-.5ex>[r] & (L,\alpha_{L}) \ar[r] \ar@<-.5ex>[l] & 0
}
\]
where $\mbox{}^{\delta}E \cong (M' \rtimes_{s'} E) / T$ and $T = \{ (\delta(m), -i(m)) \mid m \in M \}$. The following computation shows that $T$ is indeed an ideal:
\begin{align*}
[(m', e'), (\delta(m), -i(m))] &= (s'(e') \centerdot \delta(m) + s'(i(m))\centerdot m', -[e', i(m)]) \\
{} &= (\pi s'(e')\centerdot \delta(m) + \pi s'i(m) \centerdot m', -[\sigma \pi (e'), i(m)]) \\
{} &= (\delta(s\pi(e')\centerdot m), -i(s\pi(e') \centerdot m)) \in T.
\end{align*}
Therefore, $\delta$ induces a map $\delta_{*}\colon \Ext_{s}(L, M) \to \Ext_{s'}(L, M')$. 

\begin{lemma}\label{L:universal_map}
	Let $(E)$ be an abelian $s$-extension.
	\begin{itemize}
		\item[(a)] The backward induced extension given by $\gamma \colon (L',\alpha_{L'}) \to (L,\alpha_{L})$ is split if and only if there exists a homomorphism of Hom-Lie algebras $\psi \colon (L',\alpha_{L'}) \to (E,\alpha_{E})$ such that $\pi \circ \psi = \gamma$. 
		\item[(b)] The forward induced extension given by $\delta \colon (M,\alpha_{M}) \to (M',\alpha_{M'})$ is split if and only if there is an $s$-derivation $\psi' \colon  (E,\alpha_{E}) \to (M',\alpha_{M'})$ such that $\psi' \circ i = \delta$.
	\end{itemize}
\end{lemma}

\begin{proof}
	Part (a) is a direct consequence of $E_{\gamma}$ being a pullback. An easy computation checks part (b).
\end{proof}

Let $(E)$ and $(E')$ be two $s$-abelian extensions. We define their direct sum which comes equipped with a diagonal and a codiagonal map:
\[
\xymatrix{
	& & & L \ar[d]^{\bigtriangleup} & \\
	0 \ar[r] & M \oplus M \ar[d]^{\bigtriangledown} \ar[r]^-{i \oplus i'} & E \oplus E' \ar@<-.5ex>[r]_-{\pi \oplus \pi'} & L \oplus L \ar[r] \ar@<-.5ex>[l]_-{\sigma \oplus \sigma'} & 0 \\
	& M & & & 
}
\]
In this way, we introduce a sum operation as:
\[
(E) + (E') = \mbox{}^{\bigtriangledown}\big( (E \oplus E')_{\bigtriangleup} \big).
\]
Note that the order of applying $\bigtriangleup$ and $\bigtriangledown$ does not play any role, and this operation is associative and commutative. The zero element is the class of the trivial extension
\[
\xymatrix{
	0 \ar[r] & (M,\alpha_M) \ar[r]^-{i} & (M \rtimes_{s} L,\widetilde{\alpha}) \ar@<-.5ex>[r]_-{\pi} & (L,\alpha_L) \ar[r] \ar@<-.5ex>[l]_-{\sigma} & 0
}
\]
Moreover, for any scalar $k \in \K$,  we denote by $k_M$ the endomorphism of $M$ sending $m$ to $km$. Then we define $k \cdot (E)$ = $(^{k_{M}}E)$. In this way, we have introduced a vector space structure in $\Ext_{s}(M, L)$.

Also recall that for any short exact sequence of Hom-Lie algebras $(E)$
we can form its \emph{abelianisation} $\ab(E)$ in a natural way
\[
\xymatrix{
	(E) & 0 \ar[r] & (N, \alpha_{N}) \ar[d] \ar[r]^-{\xi} & (E, \alpha_{E}) \ar[d] \ar[r]^-{\pi} & (L, \alpha_{L}) \ar[d] \ar[r] & 0 \\
	(\ab(E)) &0 \ar[r] & (N^{\ab}, \alpha_{N^{\ab}}) \ar[r]^-{\xi'} & \bigg(\dfrac{E}{[N, N]}, \alpha_{\frac{E}{[N, N]}}\bigg) \ar[r]^-{\pi'} & (L, \alpha_{L}) \ar[r] & 0
}
\]

\begin{theorem}
	Let $\xymatrix@C=1em{0 \ar[r] & (N, \alpha_{N}) \ar[rr]^{\xi} && (E, \alpha_{E}) \ar@<-0.5ex>[rr]_{\pi} && (L, \alpha_{L}) \ar@<-0.5ex>[ll]_-{\sigma} \ar[r] & 0 }$ be a short exact sequence of Hom-Lie algebras where $\sigma$ is a Hom-linear section of $\pi$. Let $s= \alpha^n$ a multiple of $\alpha$, where $n \geq 0$. To every Hom-$L$-module $(A, \alpha_A)$ it corresponds a five-term natural exact sequence of Hom-vector spaces:
	\begin{equation}
	\begin{tikzpicture}[baseline=(current  bounding  box.center)]
	\matrix(m) [matrix of nodes,row sep=1em, column sep=2em, text height=2.8ex, text depth=1.5ex]
	{
		$0$  & $\Der_{s}(L, A)$  & $\Der_{s}(E, A)$ & \mbox{}  \\ 
		\mbox{}	& $\Hom(N^{\ab}, A)$  & $\Ext_{s}(L, A)$ & $\Ext_{s}(E, A)$ \\};
	\path[overlay,->, font=\scriptsize,>=angle 90]	
	(m-1-1) edge (m-1-2)
	(m-1-2) edge node[above] {$\Der_{s}(\pi)$} (m-1-3)
	(m-1-3) edge [out=355,in=175] node[pos=0.5,sloped,above] {$\zeta$} 	(m-2-2)
	(m-2-2) edge node[above] {$\vartheta^{*}$} (m-2-3)
	(m-2-3) edge node[above] {$\pi^{*}$}(m-2-4)
	;
	\end{tikzpicture}
	\end{equation}
	where $\zeta(d)(\bar{n}) = d\circ \xi(\bar{n})$ and $\vartheta^{*}(f) = f_{*}(\ab(E))$.
\end{theorem}

\begin{proof}
	First of all, the commutativity required of $s$ in the previous definitions is automatically satisfied by the properties of Hom-Lie algebras.
	Exactness in $\Der_{s}(L, A)$ is straightforward. Moreover, for any $d \in \Der_{s}(E, A)$, we have that
	\begin{align*}
	d([\xi(n), \xi(n')]) &= \xi(n) \centerdot d(\xi(n')) - \xi(n') \centerdot d(\xi(n)) \\
	{} &= \pi \xi(n) \centerdot d(\xi(n')) - \pi \xi(n') \centerdot d(\xi(n)) = 0.
	\end{align*}
	Then, $\zeta$ is well defined and our sequence is clearly exact in $\Der_{s}(E, A)$. Exactness in $\Hom(N^{\ab}, A)$ is an immediate consequence of Lemma~\ref{L:universal_map}~(b). Let us check now that $\Ima (\vartheta^{*}) = \Ker (\pi^{*})$. Given any Hom-linear homomorphism $f \colon N^{\ab} \to A$, we can consider a morphism from $(E)$ to $(\mbox{}^{f}E)$, so Lemma~\ref{L:universal_map}~(a) implies that $\Ima (\vartheta^{*}) \subseteq \Ker (\pi^{*})$. 
	Conversely, given an extension 
	\[
	\xymatrix{0 \ar[r] & (A,\alpha_A) \ar[r]^{\xi'} & (E',\alpha_{E'}) \ar@<-.5ex>[r]_-{\pi'} & (L,\alpha_L) \ar[r] \ar@<-.5ex>[l]_-{\sigma'} & 0}
	\]
	in the kernel of $\pi^{*}$, Lemma~\ref{L:universal_map}~(a) implies that there exists a Hom-Lie algebra morphism $\psi \colon  (E,\alpha_{E}) \to  (E',\alpha_{E'})$ such that $\pi = \pi' \circ \psi$. This morphism yields a map $\psi' \colon N^{\ab} \to A$ and it is easy to see that $\vartheta^{*}(\psi') = (E')$.
\end{proof}

Fixing $s$ and all $\alpha$ as identity maps, we recover a classical result in Lie algebras~\cite{Wei}. Moreover, fixing just $s$ as the identity map, we obtain the following corollary:
\begin{corollary}
	Let $\xymatrix@C=1em{0 \ar[r] & (N, \alpha_{N}) \ar[rr]^{\xi} && (E, \alpha_{E}) \ar@<-0.5ex>[rr]_{\pi} && (L, \alpha_{L}) \ar@<-0.5ex>[ll]_-{\sigma} \ar[r] & 0 }$ be a short exact sequence of Hom-Lie algebras where $\sigma$ is a Hom-linear section of $\pi$. To every Hom-$L$-module $(A, \alpha_A)$ it corresponds a five-term natural exact sequence of Hom-vector spaces:
	\begin{equation}
	\begin{tikzpicture}[baseline=(current  bounding  box.center)]
	\matrix(m) [matrix of nodes,row sep=1em, column sep=2em, text height=2.8ex, text depth=1.5ex]
	{
		$0$  & $\Der(L, A)$  & $\Der(E, A)$ & \mbox{}  \\ 
		\mbox{}	& $\Hom(N^{\ab}, A)$  & $\Ext(L, A)$ & $\Ext(E, A)$ \\};
	\path[overlay,->, font=\scriptsize,>=angle 90]	
	(m-1-1) edge (m-1-2)
	(m-1-2) edge node[above] {$\Der(\pi)$} (m-1-3)
	(m-1-3) edge [out=355,in=175] node[pos=0.5,sloped,above] {$\zeta$} 	(m-2-2)
	(m-2-2) edge node[above] {$\vartheta^{*}$} (m-2-3)
	(m-2-3) edge node[above] {$\pi^{*}$}(m-2-4)
	;
	\end{tikzpicture}
	\end{equation}
\end{corollary}

\subsection{Fixing \texorpdfstring{$s = \alpha$}{s = alpha}} \hfill

Let $(L,\alpha_{L})$ be a Hom-Lie algebra, $(M,\alpha_{M})$ a Hom-$L$-module and
$w\in Z_{\alpha}^{2}(L, M) = \Ker(d^{2})$, i.e.\ a 2-cocycle, then we can construct the abelian $\alpha$-extension
of $(L,\alpha_{L})$ by $(M,\alpha_{M})$
\begin{equation} \label{extension}
\xymatrix{
0 \ar[r] & (M,\alpha_{M}) \ar[r]^-{\chi} & (M\oplus_{w} L,\widetilde{\alpha}) \ar@<-.5ex>[r]_-{\pi} &  (L,\alpha_{L}) \ar@<-.5ex>[l]_-{\sigma} \ar[r] & 0
}
\end{equation}
where $\sigma(l) = (0, l)$ and the bracket in $(M\oplus_{w}L,\widetilde{\alpha})$ is given by
\[
[(m, l),(m', l')] = (\alpha_L(l) \centerdot m'- \alpha_L(l') \centerdot m + \omega(l, l'),\left[
l, l' \right])
\]
and $\widetilde{\alpha} (m,l) = (\alpha_{M}(m) ,\alpha_{L}(l))$.

In order to have a Hom-Lie algebra structure on $(M\oplus_{w}L,[-,-],\widetilde{\alpha})$, the Hom-Jacobi identity must be satisfied. This is indeed true whenever $w$ satisfies the following identity:
\begin{align} \label{cociclo Lie}
&\alpha_{L}^{2}(l) \centerdot w(l^{\prime
}, l^{\prime\prime}) +w(\alpha
_{L}(l) ,\left[ l^{\prime},l^{\prime\prime}\right])\nonumber\\
+&\alpha_{L}^{2}(l^{\prime\prime}) \centerdot w(
l, l^{\prime})  +w(\alpha
_{L}(l^{\prime\prime}) ,\left[ l,l^{\prime}\right])  \\
+&\alpha_{L}^{2}(l^{\prime}) \centerdot w(
l^{\prime\prime} , l )
+w(\alpha_{L}(l^{\prime}) ,\left[ l^{\prime\prime
},l\right]) = 0. \nonumber
\end{align}

Now we show that the extension (\ref{extension}) is an abelian $\alpha$-extension.
First of all, the following diagram is commutative:
\[
\xymatrix{
	M \ar[r]^-{\chi} \ar[d]_-{\alpha_M} & M\oplus_{w}L \ar[r]^-{\pi} \ar[d]^-{\widetilde{\alpha}}
	& L \ar[d]^-{\alpha_L} \\
	M \ar[r]^-{\chi}& M\oplus_{w}L \ar[r]^-{\pi}
	& L}
\]
\[
\widetilde{\alpha}({\chi} (m))
=\widetilde{\alpha}(m,0) =(\alpha_{M}(m)
,\alpha_{L}(0)) =(\alpha_{M}(m)
,0) ={\chi} (\alpha_{M}(m))
\]
\[
\alpha_{L}(\pi(m,l)) =\alpha_{L}(l) =\pi(\alpha_{M}(m) ,\alpha_{L}(l)) = \pi(\widetilde{\alpha}(m,l))
\]

Secondly, the Hom-$L$-module structure of $(M,\alpha_{M})$ induced by the extension
coincides with the Hom-$L$-module structure of $(\alpha_{L}(L), \alpha_{L\mid})$ on
$(M,\alpha_{M})$. Namely, the action induced by the extension is: 
\[
[\sigma(l), \chi(m)]=\left[ (0,l) ,(m,0)
\right] =(\alpha_L(l) \centerdot m-\alpha_L(0) \centerdot 0+w(l,0) ,\left[ l,0\right])
\equiv\alpha_{L}(l) \centerdot m
\]
Finally, $(M,\alpha_{M})$ is an abelian Hom-Lie algebra since:
\[
\left[ m,m^{\prime}\right] =
\left[ (m,0),(m^{\prime},0) \right] =
(\alpha_L(0)\centerdot m^{\prime} - \alpha_L(0) \centerdot m+w(0,0), \left[ l,0\right]) =(0,0)
\]
and $\pi$ has a linear Hom-section $\sigma \colon L \to M \oplus_{\omega} L, \sigma(l)=(0,l)$.

Moreover, the following computation shows that any bilinear map $w\colon L \wedge L \to L$ satisfying equation~\eqref{cociclo Lie} is a 2-cocycle:
\begin{align*}
d^2(w (l\wedge l^{\prime
}\wedge l^{\prime\prime})) &=\alpha_{L}^{2}(l)\centerdot
w(l^{\prime}\wedge l^{\prime\prime})-\alpha_{L}^{2}(l^{\prime})\centerdot
w(l\wedge l^{\prime\prime})+\alpha_{L}^{2}(l^{\prime\prime})\centerdot
w(l\wedge l^{\prime}) \\
{} &\quad -w(\left[ l,l^{\prime}\right] \wedge
\alpha_{L}(l^{\prime\prime})) - w(\left[
l,l^{\prime\prime}\right] \wedge\alpha_{L}(l^{\prime}))
+w(\left[ l^{\prime},l^{\prime\prime}\right] \wedge\alpha_{L}(
l)) \\
{} &= 0.
\end{align*}

\begin{proposition} Let $(L,\alpha_{L})$ be a Hom-Lie algebra and $(M,\alpha_{M})$ a
	Hom-$L$-module. Every class of abelian $\alpha$-extensions in $\Ext_{\alpha}(L,M)$ can be
	represented by an abelian $\alpha$-extension of the form
	\[
	\xymatrix{
		0 \ar[r] & (M,\alpha_{M}) \ar[r]^-{\chi} & (M\oplus_{w} L,\widetilde{\alpha}) \ar@<-.5ex>[r]_-{\pi} &  (L,\alpha_{L}) \ar@<-.5ex>[l]_-{\sigma} \ar[r] & 0.
	}
	\]
\end{proposition}

\begin{proof}
	Let $(E)$ be the following abelian $\alpha$-extension
	\begin{equation*}
	\xymatrix{0 \ar[r] & (M,\alpha_M) \ar[r]^{i} & (E,\alpha_E) \ar@<-.5ex>[r]_{\pi'} & (L,\alpha_L) \ar[r] \ar@<-.5ex>[l]_-{\sigma'} & 0}.
	\end{equation*}
	We define the $\mathbb{K}$-bilinear map 
	\[
	w\colon L \wedge L\to M
	\]
	by
	\[
	w(x_{1}\wedge
	x_{2}) = i^{-1}([ \sigma'(x_{1}),\sigma'(x_{2})] -\sigma'[ x_{1},x_{2}]).
	\]
	A routine calculation shows that $w \in Z^2_{\alpha}(L,M)$. Moreover $w$ satisfies equation~\eqref{cociclo Lie} thanks to the 2-cocycle condition. Consequently, we can construct the abelian \mbox{$\alpha$-extension}
	$\xymatrix@C=1em{0 \ar[r] & (M,\alpha_{M}) \ar[rr]^-{\chi} && (M\oplus_{w} L,\widetilde{\alpha}) \ar@<-.5ex>[rr]_-{\pi} &&  (L,\alpha_{L}) \ar@<-.5ex>[ll]_-{\sigma} \ar[r] & 0}$, which is equivalent to $(E)$ since the
	Hom-Lie algebras homomorphism ${\Phi\colon (M\oplus_{w}L,\widetilde{\alpha}) \to (E,\alpha_{E})}$ given by $\Phi\colon M\oplus_{w}L\longrightarrow E$, $\Phi (m,l) = i(
	m) +\sigma'(l)$ makes the following diagram commutative:
	\[
	\xymatrix{
		0 \ar[r]& (M,\alpha_{M}) \ar@{=}[d] \ar[r]^{\chi \quad}	& (M\oplus_{w}L,\widetilde{\alpha}) \ar@{-->}[d]_{\Phi} \ar@<-.5ex>[r]_-{\pi} & (L,\alpha_{L}) \ar@{=}[d] \ar@<-.5ex>[l]_-{\sigma} \ar[r] &0\\
		0 \ar[r] & (M,\alpha_{M}) \ar[r]^{i} & (E,\alpha_{E}) \ar@<-.5ex>[r]_-{\pi'} & (L,\alpha_{L}) \ar[r] \ar@<-.5ex>[l]_-{\sigma'} & 0
	}
	\]
\end{proof}

\begin{proposition} Two abelian $\alpha$-extensions
	\[
	\xymatrix{0 \ar[r] & (M,\alpha_{M}) \ar[r]^-{\chi} & (M\oplus_{w} L,\widetilde{\alpha}) \ar@<-.5ex>[r]_-{\pi} &  (L,\alpha_{L}) \ar@<-.5ex>[l]_-{\sigma} \ar[r] & 0}
	\]
	and
	\[
	\xymatrix{0 \ar[r] & (M,\alpha_{M}) \ar[r]^-{\chi'} & (M\oplus_{w'} L,\widetilde{\alpha}') \ar@<-.5ex>[r]_-{\pi'} &  (L,\alpha_{L}) \ar@<-.5ex>[l]_-{\sigma'} \ar[r] & 0}
	\]
 in $\Ext_{\alpha}(L,M)$ are equivalent if and only if $w$ and
	$w^{\prime}$ are cohomologous.
\end{proposition}
\begin{proof}
	Assume that there exists a homomorphism $\phi\colon (M\oplus_{w}L,\widetilde{\alpha}) \to (
	M\oplus_{w^{\prime}}L,\widetilde{\alpha}^{\prime})$ such that the following diagram is commutative:
	\[
	\xymatrix{
		0 \ar[r]& (M,\alpha_{M}) \ar@{=}[d] \ar[r]^-{\chi}
		& (M\oplus_{w}L,\widetilde{\alpha})
		\ar@{-->}[d]_-{\phi} \ar@<-.5ex>[r]_-{\pi} & (L,\alpha_{L}) \ar@<-.5ex>[l]_-{\sigma} \ar@{=}[d] \ar[r] &0\\
		0 \ar[r]& (M,\alpha_{M}) \ar[r]^-{\chi'}
		& (M\oplus_{w'} L,\widetilde{\alpha}')
		\ar@<-.5ex>[r]_-{\pi'} & (L,\alpha_{L}) \ar@<-.5ex>[l]_-{\sigma'} \ar[r] &0}
	\]
	
	The $\K$-linear map $\phi$ is necessarily of the form $\phi(m,l)
	=(m+\theta(l) ,l)$ where $\theta\colon L\to M$ is a $\mathbb{K}$-linear map.
	Moreover, since $\phi$ is a homomorphism of Hom-Lie algebras, then the following computations
	\begin{align*}
	\phi(\widetilde{\alpha}(m,l)
	) &= \phi(\alpha_{M}(m) ,\alpha_{L}(l)
	) =(\alpha_{M}(m) +\theta (\alpha_{L}(
	l)) ,\alpha_{L}(l)) \\
	\widetilde{\alpha}^{\prime}(\phi(
	m,l)) &=\widetilde{\alpha}^{\prime}(m+\theta (
	l) ,l) =(\alpha_{M}(m) +\alpha_{M}(
	\theta (l)) ,\alpha_{L}(l))
	\end{align*}
	imply that $\theta \circ \alpha_{L}=\alpha_{M} \circ \theta$.
	
	On the other hand, we have that
	\begin{align*}
	\phi\left[ (m,l) ,(m^{\prime},l^{\prime}) \right] &=\phi(\alpha_{L}(l) \centerdot m^{\prime}-\alpha_{L}(l^{\prime}) \centerdot m+w(l \wedge
	l^{\prime}) ,\left[ l,l^{\prime}\right])  \\
	{} &=(\alpha_{L}(l) \centerdot
	m^{\prime}-\alpha_{L}(l^{\prime}) \centerdot m+w(l \wedge l^{\prime
	}) +\theta \left[ l,l^{\prime}\right] ,\left[ l,l^{\prime}\right]) \\
	\\
	\left[ \phi(m,l) ,\phi(m^{\prime}%
	,l^{\prime}) \right] &=\left[ (m+\theta(l)
	,l) ,(m^{\prime}+\theta (l^{\prime}) ,l^{\prime
	}) \right] \\
	{} &= (\alpha_{L}(l)
	\centerdot (m^{\prime}+\theta (l^{\prime})) -\alpha
	_{L}(l^{\prime}) \centerdot (m+\theta (l))
	+w^{\prime}(l \wedge l^{\prime})
	,\left[ l,l^{\prime}\right]) 
	\end{align*}
	Both expressions coincide if and only if
	\begin{equation} \label{coborde Lie}
	w(l \wedge l^{\prime}) -w^{\prime}(l \wedge l^{\prime
	}) =\alpha_{L}(l) \centerdot \theta (l^{\prime})
	-\alpha_{L}(l^{\prime}) \centerdot \theta(l)
	-\theta\left[ l,l^{\prime}\right]
	\end{equation}
	for all $l, l' \in L$.
	
	From equation~\eqref{coborde Lie} we derive that $w(l \wedge l^{\prime
	}) -w^{\prime}(l \wedge l^{\prime}) =d^{1}\circ \theta (
	l,l^{\prime})$, hence $w-w^{\prime}\in BL_{\alpha}^{1}(L, M)
	=\Ima (d^{1})$.

	Conversely, if $w$ and $w^{\prime}$ are cohomologous, then there exists a $\mathbb{K}$-linear map $\theta \colon L \to M$ such that
	$\theta \circ \alpha_{L}=\alpha_{M} \circ \theta$ and $w-w^{\prime}=d^1 \circ \theta$.
	
	If we define $\phi\colon (M\oplus_{w}L, \widetilde{\alpha})\to (M\oplus_{w^{\prime}}L, \widetilde{\alpha'})$ by
	$\phi (m,l) =(m+\theta(l) ,l)$, then $\phi$ is a homomorphism of Hom-Lie algebras making the above diagram commutative.
\end{proof}

\begin{theorem} Let $(L,\alpha_{L})$ be a Hom-Lie algebra and $(M,\alpha_{M})$ a
	Hom-$L$-module. Then there exists a bijection
	\[
	\Ext_{\alpha}(L,M) \cong H_{\alpha}^{2}(L,M).
	\]
	The zero is represented by the canonical $\alpha$-abelian extension
	\[
	\xymatrix@C=1em{0 \ar[r] & (M,\alpha_M) \ar[rr]^-{i} && (M \rtimes_{\alpha} L, \widetilde{\alpha}) \ar@<-.5ex>[rr]_-{\pi'} && (L,\alpha_L) \ar[r] \ar@<-.5ex>[ll]_-{\sigma'} & 0}
	\]
	
\end{theorem}

\begin{corollary}
	Let $\xymatrix@C=1em{0 \ar[r] & (N, \alpha_{N}) \ar[rr]^{\xi} && (E, \alpha_{E}) \ar@<-0.5ex>[rr]_{\pi} && (L, \alpha_{L}) \ar@<-0.5ex>[ll]_-{\sigma} \ar[r] & 0 }$ be a short exact sequence of Hom-Lie algebras where $\sigma$ is a Hom-linear section of $\pi$. To every Hom-$L$-module $(A, \alpha_A)$ it corresponds a five-term natural exact sequence of Hom-vector spaces:
	\begin{equation*}
	\begin{tikzpicture}[baseline=(current  bounding  box.center)]
	\matrix(m) [matrix of nodes,row sep=1em, column sep=2em, text height=2.8ex, text depth=1.5ex]
	{
		$0$  & $\Der_{\alpha}(L, A)$  & $\Der_{\alpha}(E, A)$ & \mbox{}  \\ 
		\mbox{}	& $\Hom(N^{\ab}, A)$  & $H^2_{\alpha} (L, A)$ & $H^2_{\alpha} (E, A)$ \\};
	\path[overlay,->, font=\scriptsize,>=angle 90]	
	(m-1-1) edge (m-1-2)
	(m-1-2) edge node[above] {$\Der(\pi)$} (m-1-3)
	(m-1-3) edge [out=355,in=175] node[pos=0.5,sloped,above] {$\zeta$} 	(m-2-2)
	(m-2-2) edge node[above] {$\vartheta^{*}$} (m-2-3)
	(m-2-3) edge node[above] {$\pi^{*}$}(m-2-4)
	;
	\end{tikzpicture}
	\end{equation*}
\end{corollary}

\begin{theorem}
	Let $(F, \alpha_F)$ be a free Hom-Lie algebra (Section~\ref{free}). Then, for any Hom-$F$-module $(M, \alpha_{M})$, we have that  ${H_{\alpha}^2(F, M) = 0}$.
\end{theorem}

\begin{proof}
	Let $(E)\colon \xymatrix@C=0.87em{0 \ar[r] & (M,\alpha_M) \ar[rr]^{i} && (E,\alpha_E) \ar@<-.5ex>[rr]_{\pi} && (F,\alpha_F) \ar[r] \ar@<-.5ex>[ll]_{\sigma} & 0,}$ be an $\alpha$-extension of $(F, \alpha_F)$. Then, by the properties of free Hom-Lie algebras, there exists a homomorphism from $(F,\alpha_F)$ to $(M \rtimes_{\alpha} F, \widetilde{\alpha})$, inducing a homomorphism from $(E,\alpha_E)$ to $(M \rtimes_{\alpha} F, \widetilde{\alpha})$. Therefore, every $\alpha$-extension is equivalent to the trivial one.
\end{proof}

\section{Crossed modules}\label{S:crossed}
Crossed modules of Hom-Lie algebras were introduced in \cite{ShCh} in order to prove the existence of a
one-to-one correspondence between strict Hom-Lie 2-algebras and crossed modules of Hom-Lie algebras. In this section we analyse some categorical properties of crossed modules of Hom-Lie algebras.

\begin{definition} \label{def cm}
A \emph{crossed module} of Hom-Lie algebras is a triple of the form $((M, \alpha_M),(L, \alpha_L),\mu)$,
where $(M, \alpha_M)$ and $(L, \alpha_L)$ are Hom-Lie algebras together with a Hom-action from $(L, \alpha_L)$ over $(M, \alpha_M)$ and  a Hom-Lie algebra homomorphism $\mu \colon (M, \alpha_M) \to (L, \alpha_L)$ such that the following identities hold:
\begin{enumerate}
\item[(a)] $\mu(l \centerdot m) = [l, \mu(m)]$,
\item[(b)] $\mu(m) \centerdot m' = [m,m']$,
\end{enumerate}
for all $m \in M, l \in L$.
\end{definition}

\begin{remark} \label{cm properties}
The subalgebra $\Ima (\mu)$ in this case it is also a Hom-ideal of $(L,\alpha_L)$ and that $\Ker (\mu)$ is contained in the centre of $(M,\alpha_M)$. Moreover $(\Ker (\mu),\alpha_{M \mid})$ is a Hom-$Coker (\mu)$-module.
\end{remark}

\begin{example}\label{Ex CM} \
\begin{enumerate}
\item[(a)] Let $(H,\alpha_H)$ be a Hom-ideal of a Hom-Lie algebra $(L,\alpha_L)$. Then the triple $((H,\alpha_H),(L,\alpha_L), inc)$ is a crossed module where the action is given in Example \ref{ex action}~(b). There are two particular cases which allow us to think a Hom-Lie algebra as a crossed module, namely $(H,\alpha_H)=(L,\alpha_L)$ and $(H,\alpha_H)=(0,0)$. So $((L,\alpha_L),(L,\alpha_L),\Id)$ and $((0,0),(L,\alpha_L),0)$ are crossed modules.
   \item[(b)] Let $(L,\alpha_L)$ be a Hom-Lie algebra and $(M,\alpha_M)$ be a Hom-L-module. Then $((M,\alpha_M),(L,\alpha_L),0)$ is a crossed module.
\end{enumerate}
\end{example}

\begin{proposition}
Let $(M,\alpha_M)$ and $(L,\alpha_L)$ be Hom-Lie algebras together with a Hom-action of $(L,\alpha_L)$ on $(M,\alpha_M)$ and $\mu \colon (M,\alpha_M) \to (L,\alpha_L)$. The following statements are equivalent:
\begin{enumerate}
\item[(a)] $((M,\alpha_M), (L,\alpha_L), \mu)$ is a crossed module of Hom-Lie algebras.
\item[(b)] The $\K$-linear maps $(\mu,1) \colon (M,\alpha_M) \rtimes (L,\alpha_L) \to (L,\alpha_L) \rtimes (L,\alpha_L)$ and $(1,\mu) \colon (M,\alpha_M) \rtimes (M,\alpha_M) \to (M,\alpha_M) \rtimes (L,\alpha_L)$ are homomorphisms of Hom-Lie algebras.
\end{enumerate}
\end{proposition}

\begin{definition} \label{homo cm}
Let $((M,\alpha_M),(L,\alpha_L),\mu)$ and $((M',\alpha_{M'}),(L',\alpha_{L'}),\mu')$ be crossed modules of Hom-Lie algebras. A \emph{homomorphism of crossed modules} is a pair of Hom-Lie algebra homomorphisms $f \colon (M,\alpha_M) \to (M',\alpha_{M'})$ and $\phi \colon (L,\alpha_L) \to (L',\alpha_{L'})$ such that:
\begin{enumerate}
\item[(a)] $\phi \centerdot \mu = \mu' \centerdot f$,
\item[(b)] $f(l \centerdot m) = \phi(l) \centerdot f(m)$,
\end{enumerate}
for all $m \in M, l \in L$.
\end{definition}

We will denote the category of Hom-Lie crossed modules by $\Xhomlie$. It is easy to check that this definition coincides with the general notion of crossed modules in a semi-abelian category introduced by Janelidze in~\cite{Jan}

\subsection{Equivalence with cat\texorpdfstring{$^1$}{1}-Hom-Lie algebras and crossed modules} \hfill

In this subsection we stablish the relationships between crossed modules, cat$^1$-Hom-Lie algebras and internal categories. The equivalence between crossed modules and internal categories is known for categories of groups with operations~\cite{Po} and more generally, in semi-abelian categories~\cite{Jan}. Let us note that the following definition of cat$^1$-Hom-Lie algebras is given in complete analogy with Loday's original notion of cat$^1$-groups~\cite{Lod-spaces}.

\begin{definition}
	A \emph{cat$^1$-Hom-Lie algebra} $((P, \alpha_P), (N, \alpha_N), s, t)$ consists of a Hom-Lie algebra $(P, \alpha_P)$ together with a Hom-Lie subalgebra $(N, \alpha_N)$ and two homomorphisms of Hom-Lie algebras $s, t \colon (P, \alpha_P) \to (N, \alpha_N)$ satisfying the following conditions:
	\begin{enumerate}
		\item[(a)] $s_{\mid N} = t_{\mid N} = id_N$,
		\item[(b)] $[\Ker(s), \Ker(t)]=0$.
	\end{enumerate}
	If $((P', \alpha_{P'}), (N', \alpha_{N'}),s', t')$ is another cat$^1$-Hom-Lie algebra, a \emph{morphism} of cat$^1$-Hom-Lie algebras is a homomorphism of Hom-Lie algebras $f \colon (P, \alpha_P) \to (P', \alpha_{P'})$ such that $f(N) \subseteq N'$, $s' \circ f = f \circ s$ and $t' \circ f = f \circ t$. We denote the corresponding category by $\cathomlie$.
\end{definition}

We can associate to a cat$^1$-Hom-Lie algebra a crossed module as follows. Given $((P, \alpha_P), (N, \alpha_N), s, t)$, we set  $t_{\mid \Ker(s)} \colon (\Ker (s), \alpha_{\Ker (s)}) \to (N, \alpha_N)$ where the action of $N$ on $\Ker (s)$ is given by the bracket (see Example \ref{ex action}~(b)). We claim that it is a crossed module and moreover, we have the following result:

\begin{lemma}
	The above assignment defines a functor $\mathcal{P} \colon \cathomlie \to \Xhomlie$.
\end{lemma}
\begin{proof}
	It is immediate that $\alpha_N \circ t_{\mid P} = t_{\mid P} \circ \alpha_P$. Then,
	
	$t_{\mid P}(n \centerdot p) = t [n,p] =[t(n),t(p)] = [n, t_{\mid P}(p)]$,
	
	$t_{\mid P}(p) \centerdot p' = [t(p), p'] = [p,p']$ since $p' \in \Ker(s)$ and $t(p)-p \in \Ker(t)$.
	
	In addition, it is an easy calculation to check that if $f$ is a cat$^1$-Hom-Lie algebra morphism,  $\mathcal{P}(f) = (f_{\mid P}, f_{\mid N})$ is a homomorphism of crossed modules and $\mathcal{P}$ is a functor.
\end{proof}

\begin{proposition}
	There is a functor $\mathcal{S} \colon \Xhomlie \to \cathomlie$.
\end{proposition}
\begin{proof}
	Let $((N, \alpha_N),(P, \alpha_P),\mu)$ be a crossed module. We can build the assignment $\mathcal{S}((N, \alpha_N), (P, \alpha_P),\mu) = ((N \rtimes P, \widetilde{\alpha}), P, s, t)$, where $s(n,p)=p, t(n,p)=\mu(n)+p$. Now let us check that it is a cat$^1$-Hom-Lie algebra:
	\begin{align*}
		t[(n,p),(n',p')] &= t([n,n'] + p \centerdot n' - p' \centerdot n, [p,p']) \\
		{} &= \mu[n,n'] + \mu(p \centerdot n') - \mu(p' \centerdot n) + [p,p'] \\
		{} &= [\mu(n),\mu(n')] + [\mu(n),p'] + [p, \mu(n')] + [p,p'] \\
		{} &= [\mu(n)+p, \mu(n')+p'] = [t(n,p), t(n',p')], \\
		\\
		\alpha_P \circ s (n,p) &= \alpha_P(p)= s(\alpha_N(n), \alpha_P(p))= s \circ \widetilde{\alpha}(n,p), \\
		\\
		\alpha_P \circ t (n,p) &= \alpha_P(\mu(n)+p) = \mu(\alpha_N(n))+\alpha_P(p) = t(\alpha_N(n), \alpha_P(p)) \\
		{} &= t \circ \widetilde{\alpha}(n,p), \\
		\\
		[(n,0),(n',-\mu(n'))] &= ([n,n']+\mu(n') \centerdot n, 0) \\
		{} & = ([n,n']+ \mu(n')\centerdot n,0) = (0,0). \\
	\end{align*}
	Finally, if $(f_1,f_2) \colon ((N, \alpha_N),(P, \alpha_P),\mu) \to ((N', \alpha_{N'}),(P', \alpha_{P'}),\mu')$ is a homomorphism of crossed modules, then $\mathcal{S}(f_1, f_2)$ is given by $(f_1, f_2) \colon N \rtimes P \to N \rtimes P$, $(f_1,f_2)(n,p) = (f_1(n), f_2(p))$, and is easy to check that $\mathcal{S}$ is a functor.
\end{proof}

\begin{theorem}
	There is an equivalence of categories $\Xhomlie \cong \cathomlie$.
\end{theorem}

\begin{proof}
	It is easy to see that $\mathcal{P}$ and $\mathcal{S}$ are inverse of each other.
\end{proof}

\section{\texorpdfstring{$\alpha$}{Alpha}-Crossed extensions and third cohomology}\label{S:third_cohomology}

The goal of this section is to relate the third cohomology of Hom-Lie algebras by means of a particular kind of crossed modules.
Nevertheless, the standard notion of crossed modules needs to be modified to be coherent.

\begin{definition} 
	An \emph{$\alpha$-crossed module} of Hom-Lie algebras is a triple of the form $((M, \alpha_M),(L, \alpha_L),\mu)$,
	where $(M, \alpha_M)$ and $(L, \alpha_L)$ are Hom-Lie algebras together with a Hom-action from $(L, \alpha_L)$ over $(M, \alpha_M)$ and a Hom-Lie algebra homomorphism $\mu \colon (M, \alpha_M) \to (L, \alpha_L)$ such that the following identities hold:
	\begin{enumerate}
		\item[(a)] $\mu(\alpha_L(l) \centerdot m) = [l, \mu(m)]$,
		\item[(b)] $\mu(\alpha(m)) \centerdot m' = [m,m']$,
	\end{enumerate}
	for all $m \in M, l \in L$.
\end{definition}

\begin{remark} \label{alfa cm}
	Obviously, an  $\alpha$-crossed module is a crossed module when $\alpha_L = Id_L$ and $\alpha_M = Id_M$, but in general this is not true as the following example shows:
	consider the two-dimensional abelian Hom-Lie algebra $(L, \alpha_L)$, where $L=\langle \{a_1,a_2\}\rangle$ and  the bracket and endomorphism $\alpha_L$ are trivial; the three-dimensional Hom-Lie algebra $(M,\alpha_M)$, where $M = \langle \{b_1,b_2,b_3 \}\rangle$, and the bracket and $\alpha_M=0$; the Hom-Lie action from $(L,\alpha_L)$ over $(M,\alpha_M)$ given by $a_1 \centerdot b_2 = - a_2 \centerdot b_1 =  b_2$ and the homomorphism of Hom-Lie algebras $\mu\colon (M,\alpha_M) \to (L,\alpha_L)$ given by $\mu(b_1)=a_1$, $\mu(b_2) = a_2$, $\mu(b_3)=0$. Then $((M,\alpha_M), (L,\alpha_L), \mu)$ is an $\alpha$-crossed module which is not a crossed module.
	
	Moreover, consider a Hom-ideal $(M, \alpha_M)$ of a Hom-Lie algebra $(L,\alpha_L)$. According to Example~\ref{Ex CM}(a), $((M, \alpha_M), (L,\alpha_L), inc)$ is a crossed module which in general is not an $\alpha$-crossed module
	
	On the other hand, $\alpha$-crossed modules satisfy the properties given in Remark~\ref{cm properties} for crossed modules and a homomorphism of $\alpha$-crossed modules is given in Definition~\ref{homo cm} where crossed module must be  changed by $\alpha$-crossed module.
\end{remark}

\begin{example}\
	\begin{enumerate}
		\item[(a)] Let $(M,\alpha_M)$ be a Hom-ideal of a Hom-Lie algebra $(L,\alpha_L)$ satisfying that ${l-\alpha_L(l)\in Z(M)}$, for all $l \in L$. Then  $((M,\alpha_M),(L,\alpha_L),inc)$ is an $\alpha$-crossed module where the action is given in Example \ref{ex action}~(b).
		
		\item[(b)]  A Hom-Lie algebra can be seen as an $\alpha$-crossed module by $((0,0),(L, \alpha_L),0)$.
		
		\item[(c)]  Let $(L,[-,-]_L,\alpha_L)$ be a Hom-Lie algebra and $(M,\alpha_M)$ be a Hom-L-module. Then $((M,\alpha_M),(L,\alpha_L),0)$ is an $\alpha$-crossed module.
	\end{enumerate}
\end{example}

\begin{definition} \label{crossed extension}
	Let $(L,\alpha_L)$ be a Hom-Lie algebra and and $(M,\alpha_M)$ a Hom-$L$-module. An \emph{$\alpha$-crossed extension} of
	$(L,\alpha_L)$ by $(M,\alpha_M)$ is an exact sequence of Hom-Lie algebras
	\[
	\xymatrix{
		0 \ar[r] & (M,\alpha_M) \ar[r]^-{\chi} & (N,\alpha_N) \ar[r]^-{\mu} & (P,\alpha_P) \ar@<-.5ex>[r]_-{\pi} & (L,\alpha_L) \ar@<-.5ex>[l]_-{\sigma} \ar[r] & 0
	}
	\]
	such that $\mu\colon (N,\alpha_N) \to (P,\alpha_P)$ is an $\alpha$-crossed module, $\sigma$ is a Hom-linear section of $\pi$ such that the Hom-$L$-module structure coincides with the induced structure of Hom-$L$-module ($l \centerdot m = \sigma(l) \centerdot \chi(m)$, for any $l \in L, m \in M$) on $(M,\alpha_M)$, and $\mu_{\mid} \colon N \to \Ima (\mu)$ has at least one linear Hom-section $\rho \colon (\Ima (\mu), \alpha_{P \mid}) \to (N, \alpha_{N})$.
\end{definition}


\begin{definition} \label{def hom}
	A homomorphism between two $\alpha$-crossed extensions is a homomorphism of crossed modules $(\varphi, \phi) \colon ((N,\alpha_N), (P,\alpha_P), \mu) \to ((N',\alpha_{N'}), (P',\alpha_{P'}),\mu')$ such that the following diagram is commutative:
	\[
	\xymatrix{
		0 \ar[r] & (M,\alpha_M) \ar@{=}[d] \ar[r]^{\chi} & (N,\alpha_N) \ar[d]_{\varphi}
		\ar[r]^{\mu} &  (P,\alpha_P) \ar[d]_{\phi} \ar@<-.5ex>[r]_-{\pi} & (L,\alpha_L) \ar@<-.5ex>[l]_-{\sigma} \ar@{=}[d] \ar[r] &
		0 \\
		0 \ar[r] & (M,\alpha_M) \ar[r]^{\chi'} & (N',\alpha_{N'}) \ar[r]^{\mu'} &  (P,\alpha_P) \ar@<-.5ex>[r]_-{\pi'} & (L,\alpha_L) \ar@<-.5ex>[l]_-{\sigma'}
		\ar[r] & 0  }
	\]
\end{definition}

\begin{definition}
	Given two $\alpha$-crossed extensions of $(L,\alpha_L)$ by $(M,\alpha_M)$ as above, they are said to be
	\emph{elementary equivalent} if there is a homomorphism from one to the other.
	In the set of all $\alpha$-crossed extensions we consider the
	equivalence relation generated by the elementary equivalence. Let
	$\Ext^2_{\alpha}(L,M)$ denote the set of equivalence
	classes of crossed extensions of the Hom-Lie algebra
	$(L,\alpha_L)$ by the Hom-$L$-module $(M,\alpha_M)$.
\end{definition}
As it happens in the abelian $s$-extensions case, the choice of the sections does not play any role.

\begin{theorem}\label{main th}
	For any Hom-Lie algebra $(L,\alpha_L)$ and a Hom-$L$-module $(M,\alpha_M)$ there is a canonical assignment:
	\[
	\eta\colon \Ext^2_{\alpha}(L,M) \to H^3_{\alpha}(L,M).
	\]
\end{theorem}
\begin{proof}
	We define $\eta\colon \Ext^2_{\alpha}(L,M)\to
	H_{\alpha}^3(L,M)$ as follows. Given an $\alpha$-crossed extension
	\[
\xymatrix{
(\xi): &	0 \ar[r] & (M,\alpha_M) \ar[r]^-{\chi} & (N,\alpha_N) \ar[r]^-{\mu} & (P,\alpha_P) \ar@<-.5ex>[r]_-{\pi} & (L,\alpha_L) \ar@<-.5ex>[l]_-{\sigma} \ar[r] & 0
}
\]
	with Hom-linear section $\rho \colon (\Ima (\mu), \alpha_{P \mid}) \to (N, \alpha_{N})$. For any $l_1,l_2 \in L$ we have that $[\sigma(l_1),\sigma(l_2)]-\sigma[l_1,l_2]\in \Ker(\pi)=\Ima(\mu)$. Taking
	\[
	f(l_1,l_2)=\rho ([\sigma(l_1),\sigma(l_2)]-\sigma[l_1,l_2])
	\] we define
	\begin{equation}\label{eq7}
		\begin{split}
			h_{\xi}(l_1,l_2,l_3) = &f(\alpha_L(l_1),[l_2,l_3])+f(\alpha_L(l_2),[l_3,l_1]) \\
			{} &+ f(\alpha(l_3),[l_1,l_2])+ \sigma(\alpha^2_L(l_1)) \centerdot f(l_2,l_3) \\
			{} &+\sigma(\alpha^2_L(l_2)) \centerdot f(l_3,l_1)+ \sigma(\alpha^2_L(l_3)) \centerdot f(l_1,l_2).
		\end{split}
	\end{equation}
	Using $\alpha$-crossed modules properties, we can directly check that $\mu h_{\xi}(l_1,l_2,l_3)=0$ and therefore $h_{\xi}(l_1,l_2,l_3)\in M =\Ker(\mu)$.
	
	Thus we have defined a linear map $h_{\xi}\colon L^{\wedge 3}\to M$ such that $h \circ \alpha_L^{\wedge 3} = \alpha_M \circ h$, that is to say, $h_{\xi} \in C^3_{\alpha}(L,M)$. Routine calculations show that $d^3 (h_{\xi}) = 0$. Here it is only necessary to apply the Hom-Jacobi identity.
	Thus $h_{\xi} \in Z_{\alpha}^3(L,M)$ and we define $\eta (\xi) = h_{\xi} + B^3_{\alpha}(L,M)$.
	
	Now we are going to check that $\eta$ is a well-defined map, i.e.\ the class of $h_{\xi}$ does not depend on the Hom-sections, $\sigma$ and $\rho$, and if there is a map of $\alpha$-crossed extensions $(\varphi,\phi) \colon (\xi)\to (\xi')$, then $h_{\xi}=h_{\xi'}$ in $H_{\alpha}^3(L,M)$.
	
	Let $\overline{\sigma}\colon L\to P$ be another linear Hom-section of $\pi$ and $\overline{h}_{\xi}$ be the $3$-cocycle defined
	using $\overline{\sigma}$ instead of $\sigma$. Since $\sigma$ and $\overline{\sigma}$ are both Hom-sections of $\pi$ there exists a linear map $g\colon L \to
	N$ such that $\mu \circ g = \overline{\sigma}- \sigma$. Then, by the definition of $h_{\xi}$, $\overline{h}_{\xi}$ and $\alpha$-crossed modules properties, we easily get
	\begin{align}\label{diff h}
		(\overline{h}_{\xi}-h_{\xi})(l_1,&l_2,l_3)=  (\overline{f}-f)(\alpha_L(l_1),[l_2,l_3])+ (\overline{f}-f)(\alpha_L(l_2),[l_3,l_1]) \nonumber \\ 
		&+ (\overline{f}-f)(\alpha_L(l_3),[l_1,l_2]) + \sigma(\alpha_L^2(l_1)) \centerdot (\overline{f}-f)(l_2,l_3)\nonumber \\
		&+ \sigma(\alpha_L^2(l_2)) \centerdot (\overline{f}-f)(l_3,l_1) + \sigma(\alpha_L^2(l_3)) \centerdot (\overline{f}-f)(l_1,l_2)\nonumber \\ 
		&+ [g(\alpha_L(l_1)),\rho([\overline{\sigma}(l_2),\overline{\sigma}(l_3)]-\overline{\sigma}[l_2,l_3])]\nonumber \\  
		&+ [g(\alpha_L(l_2)),\rho([\overline{\sigma}(l_3),\overline{\sigma}(l_1)]-\overline{\sigma}[l_3,l_1])] \\
		&+ [g(\alpha_L(l_3)),\rho([\overline{\sigma}(l_1),\overline{\sigma}(l_2)]-\overline{\sigma}[l_1,l_2])]\nonumber ,
	\end{align}
	where $\overline{f}(l_1,l_2)=\rho([\overline{\sigma}(l_1),\overline{\sigma}(l_2)]-\overline{\sigma}[l_1,l_2])$.
	
	Next we define $\widetilde{f}\colon L{\wedge} L\to N$ by
	\[
	\widetilde{f}(l_1,l_2)= - [g(l_1),g(l_2)]-g[l_1,l_2]+\overline{\sigma}(\alpha_L(l_1)) \centerdot g(l_2)- \overline{\sigma}(\alpha_L(l_2)) \centerdot g(l_1).
	\]
	Since $\mu \circ \widetilde{f}=\mu \circ (\overline{f}-f)$, then
	$\overline{f}-f-\widetilde{f}$ is a map from $(L \wedge L, \alpha_{L}^{\wedge 2})$ to $(M, \alpha_M)$. Moreover, if we replace $\overline{f}-f$ by
	$\widetilde{f}$ in~\eqref{diff h}, then the equality still remains
	true in $H_{\alpha}^3(L,M)$ since the difference is
	the coboundary $d^2(\overline{f}-f-\widetilde{f})$. After
	replacing $\overline{f}-f$ by $\widetilde{f}$ in~\eqref{diff h}, using
	the Hom-Jacobi identity and $\alpha$-crossed modules properties,
	we directly obtain that $(\overline{h}_{\xi}-h_{\xi})(l_1,l_2,l_3)=0$ in
	$H^3_{\alpha}(L,M)$. Therefore the class of
	$h_{\xi}$ does not depend on the section $\sigma$.
	
	Let $(\varphi,\phi)\colon (\xi) \to (\xi')$ be a morphism of crossed $\alpha$-extensions with the notation of Definition~\ref{def hom}. Let
	$\sigma\colon (L,\alpha_L) \to (P,\alpha_P)$ and $\rho\colon (\Ima (\mu),\alpha_{P \mid})\to (N, \alpha_N)$ be linear Hom-sections of $\pi$ and $\mu$, respectively, and let $\sigma'\colon (L, \alpha_L) \to (P',\alpha_{P'})$ and $\rho'\colon (\Ima(\mu'),\alpha_{P' \mid})\to (P',\alpha_{P'})$
	be linear Hom-sections of $\pi'$ and $\mu'$, respectively. Hence, in~\eqref{eq7} we can use $\sigma$ and $\rho$ to define $h_{\xi}$, and $\sigma'$ and $\rho'$ to define $h_{\xi'}$. But since $\pi'\circ \phi \circ \sigma=id_{L}$,
	$\phi \circ \sigma$ is another linear Hom-section of $\pi'$ and therefore we can
	replace $\sigma'$ by $\phi \circ \sigma$ to define $h_{\xi'}$. Then we easily
	derive the following equality in $H^3_{\alpha}(L,M)$
	\[
	(h_{\xi}-h_{\xi'})(l_1,l_2,l_3)=d^2(\widehat{f})(l_1,l_2,l_3)
	\]
	for the linear map $\widehat{f}\colon L^{\otimes 2}\to M$ given by
	\[
	\widehat{f}(l_1,l_2)=(\varphi \circ \rho-\rho'\circ \phi)([\sigma(l_1),\sigma(l_2)]-\sigma[l_1,l_2]).
	\]
	This proves that $h_{\xi}=h_{\xi'}$ in
	$H^3_{\alpha}(L,M)$ and also the class of
	$h_{\xi}$ does not depend on the section $\rho$. Therefore
	the map $\eta$ is well defined.
\end{proof}

\begin{remark}
	The usage of $\alpha$-crossed modules instead of standard crossed modules is crucial to obtain that $h_{\xi}(l_1, l_2, l_3)$ belongs to the kernel of $\mu$ for any $l_1, l_2, l_3 \in L$. 
\end{remark}

\begin{remark}
	As in the classical case, we would like to stablish an isomorphism between $\Ext^2_{\alpha}(L,M)$ and $H^3_{\alpha}(L,M)$. Nevertheless, when we study this situation we found two fundamental obstacles. Firstly, it is not easy to construct a canonical example of $\alpha$-crossed extension for a given cohomology class. The second one is that, although we know that the second cohomology classes of a free Hom-Lie algebra are trivial, we do not know if this is still true in higher order. The veracity of this result would be very helpful to define an inverse to $\eta$ in Theorem~\ref{main th}.
	In any case, these questions will be further investigated in following articles. 
\end{remark}

\section*{Acknowledgements}
The authors would like to thank Tim Van der Linden for his helpful suggestions.


\begin{thebibliography}{10}
	
	\bibitem{Beck}
	J.~M. Beck, \emph{Triples, algebras and cohomology}, Reprints in Theory and
	Applications of Categories \textbf{2} (2003), 1--59, Ph.D. thesis, Columbia
	University, 1967.
	
	\bibitem{Borceux-Bourn}
	F.~Borceux and D.~Bourn, \emph{Mal'cev, protomodular, homological and
		semi-abelian categories}, Math. Appl., vol. 566, Kluwer Acad. Publ., 2004.
	
	\bibitem{BoJaKe2}
	F.~Borceux, G.~Janelidze, and G.~M. Kelly, \emph{Internal object actions},
	Comment. Math. Univ. Carolin. \textbf{46} (2005), no.~2, 235--255.
	
	\bibitem{CaGo}
	S.~Caenepeel and I.~Goyvaerts, \emph{Monoidal {H}om-{H}opf algebras}, Comm.
	Algebra \textbf{39} (2011), no.~6, 2216--2240.
	
	\bibitem{CaInPa}
	J.~M. Casas, M.~A. Insua, and N.~Pacheco, \emph{On universal central extensions
		of {H}om-{L}ie algebras}, Hacet. J. Math. Stat. \textbf{44} (2015), no.~2,
	277--288.
	
	\bibitem{CaVdL}
	J.~M. Casas and T.~Van~der Linden, \emph{Universal central extensions in
		semi-abelian categories}, Appl. Categ. Structures \textbf{22} (2014), no.~1,
	253--268.
	
	\bibitem{acc}
	A.~S. Cigoli, J.~R.~A. Gray, and T.~Van~der Linden, \emph{Algebraically
		coherent categories}, Theory Appl. Categ. \textbf{30} (2015), no.~54,
	1864--1905.
	
	\bibitem{GoVe}
	I.~Goyvaerts and J.~Vercruysse, \emph{A note on the categorification of {L}ie
		algebras}, Lie theory and its applications in physics, Springer Proc. Math.
	Stat., vol.~36, Springer, Tokyo, 2013, pp.~541--550.
	
	\bibitem{Gray2012}
	J.~R.~A. Gray, \emph{Algebraic exponentiation in general categories}, Appl.\
	Categ.\ Structures \textbf{20} (2012), 543--567.
	
	\bibitem{HaLaSi}
	J.~T. Hartwig, D.~Larsson, and S.~D. Silvestrov, \emph{Deformations of {L}ie
		algebras using {$\sigma$}-derivations}, J. Algebra \textbf{295} (2006),
	no.~2, 314--361.
	
	\bibitem{HiSt}
	P.~J. Hilton and U.~Stammbach, \emph{A course in homological algebra}, second
	ed., Graduate Texts in Mathematics, vol.~4, Springer-Verlag, New York, 1997.
	
	\bibitem{Huq}
	S.~A. Huq, \emph{Commutator, nilpotency, and solvability in categories}, Quart.
	J. Math. Oxford Ser. (2) \textbf{19} (1968), 363--389.
	
	\bibitem{Jan}
	G.~Janelidze, \emph{Internal crossed modules}, Georgian Math. J. \textbf{10}
	(2003), no.~1, 99--114.
	
	\bibitem{JaMaTh}
	G.~Janelidze, L.~M\'arki, and W.~Tholen, \emph{Semi-abelian categories}, J.
	Pure Appl. Algebra \textbf{168} (2002), no.~2-3, 367--386, Category theory
	1999 (Coimbra).
	
	\bibitem{KaLo}
	C.~Kassel and J.-L. Loday, \emph{Extensions centrales d'alg\`{e}bres de {L}ie},
	Ann. Inst. Fourier (Grenoble) \textbf{32} (1982), no.~4, 119--142.
	
	\bibitem{Lod-spaces}
	J.-L. Loday, \emph{Spaces with finitely many nontrivial homotopy groups}, J.
	Pure Appl. Algebra \textbf{24} (1982), no.~2, 179--202.
	
	\bibitem{MaSi}
	A.~Makhlouf and S.~Silvestrov, \emph{Notes on 1-parameter formal deformations
		of {H}om-associative and {H}om-{L}ie algebras}, Forum Math. \textbf{22}
	(2010), no.~4, 715--739.
	
	\bibitem{MaSi2}
	A.~Makhlouf and S.~D. Silvestrov, \emph{Hom-algebra structures}, J. Gen. Lie
	Theory Appl. \textbf{2} (2008), no.~2, 51--64.
	
	\bibitem{MaVa}
	N.~Martins-Ferreira and T.~Van~der Linden, \emph{A note on the ``{S}mith is
		{H}uq'' condition}, Appl. Categ. Structures \textbf{20} (2012), no.~2,
	175--187.
	
	\bibitem{Po}
	T.~Porter, \emph{Extensions, crossed modules and internal categories in
		categories of groups with operations}, Proc. Edinburgh Math. Soc. (2)
	\textbf{30} (1987), no.~3, 373--381.
	
	\bibitem{Rat}
	J.~G. Ratcliffe, \emph{Crossed extensions}, Trans. Amer. Math. Soc.
	\textbf{257} (1980), no.~1, 73--89. \MR{549155}
	
	\bibitem{Serre}
	J.-P. Serre, \emph{Lie algebras and {L}ie groups}, Lecture Notes in
	Mathematics, vol. 1500, Springer-Verlag, Berlin, 2006, 1964 lectures given at
	Harvard University, Corrected fifth printing of the second (1992) edition.
	
	\bibitem{Sheng}
	Y.~Sheng, \emph{Representations of hom-{L}ie algebras}, Algebr. Represent.
	Theory \textbf{15} (2012), no.~6, 1081--1098.
	
	\bibitem{ShCh}
	Y.~Sheng and D.~Chen, \emph{Hom-{L}ie 2-algebras}, J. Algebra \textbf{376}
	(2013), 174--195.
	
	\bibitem{Wei}
	C.~A. Weibel, \emph{An introduction to homological algebra}, Cambridge Studies
	in Advanced Mathematics, vol.~38, Cambridge University Press, Cambridge,
	1994.
	
	\bibitem{Yau}
	D.~Yau, \emph{Enveloping algebras of {H}om-{L}ie algebras}, J. Gen. Lie Theory
	Appl. \textbf{2} (2008), no.~2, 95--108.
	
	\bibitem{Yau2}
	D.~Yau, \emph{Hom-algebras and homology}, J. Lie Theory \textbf{19} (2009),
	no.~2, 409--421.
	
\end{thebibliography}

\providecommand{\bysame}{\leavevmode\hbox to3em{\hrulefill}\thinspace}
\providecommand{\MR}{\relax\ifhmode\unskip\space\fi MR }
\providecommand{\MRhref}[2]{%
	\href{http://www.ams.org/mathscinet-getitem?mr=#1}{#2}
}
\providecommand{\href}[2]{#2}

\end{document}